\documentclass[twoside]{article}

\usepackage[margin=1in]{geometry}

\title{Asymptotic Analysis for Randomly Forced MHD}
\author{J. F\"oldes, S. Friedlander, N. Glatt-Holtz, G. Richards\\
  \scriptsize{emails: juraj.foldes@ulb.ac.be, susanfri@usc.edu, negh@math.vt.edu, g.richards@rochester.edu}}
\date{}

\usepackage{amsfonts, amssymb, amsmath, amsthm, multicol,bbm, 
            enumerate}
\usepackage[colorlinks=true, pdfstartview=FitV, linkcolor=blue,
            citecolor=blue, urlcolor=blue]{hyperref}
\usepackage[usenames]{color}
\definecolor{Red}{rgb}{0.7,0,0.1}
\definecolor{Green}{rgb}{0,0.7,0}
\usepackage{accents}
\usepackage{comment}

\providecommand{\bysame}{\leavevmode\hbox to3em{\hrulefill}\thinspace}
\providecommand{\MR}{\relax\ifhmode\unskip\space\fi MR }

\providecommand{\href}[2]{#2}

\numberwithin{equation}{section}

\newtheorem{Thm}{Theorem}[section]
\newtheorem{Lem}[Thm]{Lemma}
\newtheorem{Prop}[Thm]{Proposition}
\newtheorem{Cor}[Thm]{Corollary}

\newtheorem{Rmk}[Thm]{Remark}

\newtheorem*{Thm*}{Theorem}

\newcommand{\pd}[1]{\partial_{#1}}

\newcommand{\E}{\mathbb{E}}
\newcommand{\Prb}{\mathbb{P}}

\newcommand{\eps}{\varepsilon}
\newcommand{\del}{\delta}

\newcommand{\RR}{\mathbb{R}}
\newcommand{\TT}{\mathbb{T}}
\newcommand{\ZZ}{\mathbb{Z}}

\newcommand{\U}{U}
\newcommand{\B}{B}
\newcommand{\pr}{P}
\newcommand{\Th}{\Theta}
\newcommand{\bfe}{\hat{e}}

\newcommand{\OA}{A}

\newcommand{\OQ}{Q}

\newcommand{\bl}{b}
\newcommand{\ul}{u}

\newcommand{\prl}{p}
\newcommand{\thl}{\theta}

\newcommand{\uc}{u_\eps}
\newcommand{\bc}{b_\eps}
\newcommand{\thc}{\thl_\eps}

\newcommand{\gamp}{\gamma'}

\newcommand{\mcx}{\mathcal{X}}
\newcommand{\mct}{\mathcal{T}}

\newcommand{\ucc}{u_{\eps,\delta}}
\newcommand{\bcc}{b_{\eps, \delta}}
\newcommand{\thcc}{\thl_{\eps, \delta}}

\newcommand{\BBB}{\hat{B}_0}
\newcommand{\MMM}{\mathcal{M}}
\newcommand{\WW}{\mathfrak{W}}

\newcommand{\Omm}{\hat{\Omega}}

\newcommand{\R}{\Pi}

\begin{document}
\markboth{J. F\"oldes, S. Friedlander, N. Glatt-Holtz, G. Richards}
{Asymptotic Analysis for Randomly Forced MHD}

\maketitle

\begin{abstract}
  We consider the three-dimensional magnetohydrodynamics (MHD)
  equations in the presence of a spatially degenerate stochastic forcing
  as a model for magnetostrophic turbulence in the Earth's fluid core.  We examine
  the multi-parameter singular limit of vanishing Rossby number $\eps$ and magnetic Reynold's number $\delta$, and establish
  that: (i) the limiting stochastically driven active scalar equation
  (with $\eps=\delta=0$) possesses a unique ergodic invariant measure, and (ii) any suitable sequence of statistically invariant
  states of the full MHD system converge weakly, as $\eps,\delta \rightarrow 0$, to the unique invariant measure
  of the limit equation.  This latter convergence result does not require any conditions on the relative rates at which $\eps, \delta$ decay.

  Our analysis of the limit equation relies on a recently developed theory of hypo-ellipticity for infinite-dimensional stochastic dynamical systems.  
  We carry out a detailed study of the interactions between the nonlinear and stochastic terms to demonstrate that a H\"{o}rmander
  bracket condition is satisfied, which yields a contraction property for the limit equation in a suitable Wasserstein metric.  
  This contraction property reduces the convergence of invariant states in the multi-parameter limit to the convergence of solutions at finite times.  However, in view of the phase space mismatch between the small parameter system and the limit equation, and due to the multi-parameter
  nature of the problem, further analysis is required to establish the singular limit.  In particular, we develop methods to lift the contraction for the limit equation to the extended phase space, including the velocity and magnetic fields.  Moreover, for the convergence of solutions at finite times we make use of a probabilistic modification of the Gr\"onwall inequality, relying on a delicate stopping time argument.
\end{abstract}

  {\noindent \small
    }

\setcounter{tocdepth}{1}
\tableofcontents

\newpage

\section{Introduction}

It has long been appreciated that invariant measures of the equations of fluid dynamics provide a
mathematically consistent framework for studying robust statistical quantities in turbulent flows.
 An ongoing challenge is therefore to investigate properties of
these states such as the existence,
uniqueness, ergodicity, and dependence on parameters in a variety of specific contexts.
While one may certainly pose these questions for deterministic equations (cf. \cite{FoiasManleyRosaTemam01}),
the stochastic setting, apart from its physical relevance, can be more accessible due to the regularizing effect of noise on the associated probability
distribution functions.

In this article we will analyze a stochastically forced system of equations from magnetohydrodynamics
that was presented by Moffat and Loper \cite{moffattloper1994} as a model for magnetostrophic turbulence in the
Earth's fluid core.
The stochastic forcing terms are to be interpreted as a source of heat which  ``continuously regenerates the
statistically stationary temperature distribution throughout the core"; cf. \cite{moffatt2008}.
Our analysis will focus on the convergence of such statistically stationary states in a certain singular parameter limit suggested by \cite{moffattloper1994} and which bears some formal similarities to a large Prandtl limit considered in our recent work \cite{FoldesGlattHoltzRichards2015}.
See also \cite{Wang2004a,Wang2005,Wang2007,Wang2008b, Parshad2010} for other formally analogous limits considered within a deterministic framework.

\subsection*{The Stochastic Magnetohydrodynamics (MHD) Equations}

The three-dimensional MHD equations govern the motion of a rotating, density stratified, electrically conducting fluid under the
Boussinesq approximation. See e.g. \cite{Moffatt1978, Davidson2001}. In a rotating frame of reference these equations are written in terms of dimensionless variables as
\begin{align}
  \eps( \partial_t \U + \U \cdot \nabla \U) + \Omm \times \U
 	&= - \nabla \pr + \BBB \cdot \nabla \B + \delta \B \cdot \nabla \B -  \Th \hat{g} + \nu \Delta \U, \quad \nabla \cdot \U = 0,
	\label{eq:MHD:vel}\\
 \delta( \partial_t \B + \U \cdot \nabla \B - \B \cdot \nabla \U )
 	&= \BBB \cdot \nabla \U + \Delta \B, \quad \nabla \cdot \B = 0,
	\label{eq:MHD:mag}\\
 d \Th + \U \cdot \nabla \Th  dt
	&= \kappa \Delta \Th \, dt + \sigma dW.
	\label{eq:MHD:temp}
\end{align}
The unknowns are $\U=\U(x,t)$ the velocity,
$\B=\B(x,t)$ the magnetic field (both vector-valued) and $\Th=\Th(x,t)$ the temperature (a scalar) of the fluid.
Here $\Omm$ denotes a unit vector along the axis of rotation
of a sphere, $\hat{g}$ is a unit vector in the local direction of gravity, which
points radially inwards to the sphere, and $\BBB$ denotes a unit vector in the direction of a constant
``applied'' underlying magnetic field.
We will work in Cartesian coordinates centered on a local tangent
plane to the sphere at a co-latitude angle $\lambda$.  We choose the Cartesian frame $\{\bfe_1,\bfe_2,\bfe_3\}$ such that $\hat{g} = -\bfe_3$
and $\Omm = \cos \lambda \bfe_{3} - \sin \lambda \bfe_2$.

The physical forces governing
this system are the Coriolis force, Lorentz force, and gravity acting via buoyancy, while the equation for the temperature
$\Th$ is driven by a white-in-time, spatially correlated, Gaussian noise $\sigma dW$.  Specifically we consider
  \begin{align}
     \sigma dW = \sum_{\substack{k \in \ZZ^3_0 \\ m \in \{0, 1\}}} \alpha_{k, m} \sigma_k^m dW^{k, m} \,,
     \label{eq:n:s}
  \end{align}
  where $\{\sigma_k^m\}$ is a basis of $L^2(\mathbb{T}^3)$ consisting of eigenfunctions of Laplacian on the cubic torus, namely
$\sigma_k^0 (x) := \cos(k\cdot x)$, $\sigma_k^1 (x) := \sin(k\cdot x)$, $\alpha_k \in \RR$ are amplitudes,
and $\{W^{k, m}\}$ is a collection of independent standard, 1D Brownian motions.  We are
  mainly interested in the situation when \eqref{eq:n:s} is `degenerate' in the sense that only a few selected amplitudes $\alpha_k$
  are non-zero.  Nevertheless many of the results established below remain true (and in some cases are easier to prove)
  if infinitely many frequencies are activated ($\alpha_k \not = 0$) so long as the amplitudes decay sufficiently rapidly.

The non-dimensional parameters in \eqref{eq:MHD:vel}--\eqref{eq:MHD:temp} are $\eps$ the
Rossby number, $\delta$ the magnetic Reynolds number, $\nu$ a (non-dimensional) viscosity and
$\kappa$ a (non-dimensional) thermal diffusivity.
The orders of magnitude of the non-dimensional parameters are motivated by the physical
postulates of the Moffatt and Loper model. In particular, the parameters $\eps$, $\delta$, $\nu$, and $\kappa$ are all small.
Note that the ratio of the Coriolis to Lorentz forces in the model we are studying is of order 1, so for notational simplicity we have set this parameter, denoted by $N^2$ in \cite{moffattloper1994}, equal to 1.

For mathematical tractability we consider both \eqref{eq:MHD:vel}--\eqref{eq:MHD:temp} and \eqref{eq:AS:th}--\eqref{eq:AS:mag} below
on the periodic domain $\TT^3$, with all of the fields being mean free, a condition which is preserved by the equation.\footnote{To verify this property, one
first confirms that the mean zero condition is preserved in the $\Th$ and $\B$ component, \eqref{eq:MHD:temp}, as a straightforward consequence of
Gauss' theorem, noting that $\sigma$ is taken to be mean free in \eqref{eq:n:s}.
Then, by integrating \eqref{eq:MHD:vel} and using Gauss' theorem, recalling that the buoyancy term has zero
mean as derived already, one is left with the equation $\frac{d}{dt} V = \Omm \times V$, where $V= \int_{\TT^3}U\,dx$.
But this ODE has only the trivial solution for initial condition $V(0) = 0$, yielding the preservation of the mean free condition for $\U$.}
Notice that we have followed the usual convention of including the divergence-free condition on the magnetic field $\B$ in
\eqref{eq:MHD:mag}.  This is really a condition on the initial data, $\nabla \cdot B|_{t=0} = 0$, as this property is preserved by the dynamics of \eqref{eq:MHD:vel}--\eqref{eq:MHD:temp}.

The analysis in Moffat and Loper \cite{moffattloper1994} and Moffat \cite{moffatt2008} assumes that the statistics of the temperature field $\Th$, averaged over scales characteristic of the
turbulent regions, are prescribed.  Further simplifications are made in \cite{moffattloper1994,moffatt2008} based on the assumption
that $\eps$ and $\delta$ are both small, and for their purposes can be neglected.  In our work we carry
the analysis to a higher level of complexity and examine the three-dimensional coupled system \eqref{eq:MHD:vel}--\eqref{eq:MHD:temp}
in which a back reaction of the temperature field on the flow is allowed. We give a rigorous justification of the reduction of the full evolution system in the limit $\eps, \delta \to 0$ to
a stochastically driven active scalar equation by deriving suitable asymptotics in $\eps$ and $\delta$ at both finite and infinite time.  The limit system when $\eps = \delta = 0$ is the following active scalar equation:
\begin{equation}
 d \thl + u \cdot \nabla \thl \, dt =  \kappa \Delta \thl \, dt + \sigma dW \,,
   	\label{eq:AS:th}
\end{equation}
where  $\ul, \bl$ are related to $\thl$ via
\begin{align}
 	 \Omm \times \ul &=  -\nabla \prl + \BBB\cdot \nabla \bl  + \thl \bfe_3 + \nu \Delta \ul, \quad \nabla \cdot \ul = 0,
	\label{eq:AS:vel}\\
 	 0&= \BBB\cdot \nabla \ul +\Delta \bl \,.
	\label{eq:AS:mag}
\end{align}
Observe that the vector fields $\ul$ and $\bl$ are computed from the temperature $\thl$
and are not prescribed independent initial conditions.
Using \eqref{eq:AS:vel}--\eqref{eq:AS:mag} we can define a linear operator $M = (M_\ul, M_\bl)$ with $(u, b) = (M_\ul(\theta), M_\bl(\theta))$ and
it is crucially important in our
analysis below that $M$ produces two degrees of smoothing in space.  Vector manipulations as in  \cite{moffattloper1994,FriedlanderSuen2015},
taking the curl of \eqref{eq:AS:vel} three times and using $\nabla \cdot u = 0 = \nabla \cdot b$, yield the expression:
\begin{equation}\label{eq:eod}
\{[\nu \Delta^2 - (\BBB \cdot \nabla)^2]^2 + (\Omm \cdot \nabla)^2 \Delta \} u =
- [\nu \Delta^2 - (\BBB \cdot \nabla)^2] \nabla \times (\bfe_3 \times \nabla \theta) + (\Omm \cdot \nabla) \Delta (\bfe_3 \times \nabla \theta) \,.
\end{equation}
Hence the explicit expression for the Fourier multiplier symbols $\MMM_u(k)$ and $\MMM_b(k)$ as functions of the Fourier variable
$k = (k_1, k_2, k_3) \in \ZZ^3$ are
\begin{align}\label{eq:smb:u1}
\MMM_u (k) =  \frac{1}{D(k)} \left( (\nu |k|^4 + (\BBB \cdot k)^2)(k\times (\bfe_3 \times k)) + (\Omm \cdot k)|k|^2 (\bfe_3 \times k)  \right)
,\
\MMM_b(k) = \frac{i\BBB \cdot k}{|k|^2} \MMM_u(k) \,,
\end{align}
where
\begin{equation}
\label{eq:smb:u2}
D(k) = |k|^2 (\Omm \cdot k)^2 + ((\BBB \cdot k)^2 + \nu |k|^4)^2 \,.
\end{equation}
Since $D$ is of order $|k|^8$ whereas the numerator of $\MMM_u$ is of order $|k|^6$, the two degrees of smoothing of $M$
is evident from \eqref{eq:smb:u1}--\eqref{eq:smb:u2}.

It is worth emphasizing that the singular limit of \eqref{eq:MHD:vel}--\eqref{eq:MHD:temp} to \eqref{eq:AS:th}--\eqref{eq:AS:mag} in this work
bears some significant formal similarities to the infinite Prandtl limit for the stochastic Boussinesq equations
which we have recently considered in \cite{FoldesGlattHoltzRichards2015}.  Here however our problem
involves multiple small parameters and both the linear and non-linear structure structure of the governing
equations is quite different.  On the other hand the parameters $\nu, \kappa$ in \eqref{eq:MHD:vel}--\eqref{eq:MHD:temp} are also small in
practice suggesting other challenging singular perturbations problems of interest cf. \cite{FriedlanderRusinVicol2014,FriedlanderSuen2015,FriedlanderVicol2011,FriedlanderVicol2011b,FriedlanderVicol2012}.
Moreover a variety of other interesting singular limit problems arise from \eqref{eq:MHD:vel}--\eqref{eq:MHD:temp} but most of the
rigorous work so far is restricted to a deterministic setting.  See e.g.
\cite{Wu1997, DesjardinsDormyGrenier1999,Masmoudi2007, CheminDesjardinsGallagherGrenier2006, CaoWu2011,
FeffermanMcCormickRobinsonRodrigo2014,CalkinsJulienTobiasAurnou2015}.

\subsection*{Summary of Main Results}

We now give a preview of the main results of this work.
Precise statements are given below in the body of the work.
\begin{itemize}
\item[(i)] Consider any initial conditions $\U_{\eps, \delta}(0), \B_{\eps, \delta}(0), \Th_{\eps, \delta}(0)$,
where $\U_{\eps, \delta}(0), \B_{\eps, \delta}(0)$ are uniformly bounded in $L^2(\TT^3)$ independently of $\eps, \delta > 0$
and
\begin{align}
\label{eq:i1}
   \| \Th_{\eps, \delta}(0) - \thl(0) \| \to 0, \; \textrm{ as } \eps, \delta \to 0,
\end{align}
for some $\thl(0) \in L^2(\TT^3)$.
Here and in all that follows $\|\cdot\|$ denotes the usual $L^2$-norm.
Suppose that $(\U_{\eps, \delta}, \B_{\eps, \delta}, \Th_{\eps, \delta})$ and $\theta$ are corresponding solutions of \eqref{eq:MHD:vel}--\eqref{eq:MHD:temp} and
\eqref{eq:AS:th} respectively. Then for any $t>0$, there exists a sufficiently small $\gamma>0$ such that for $p>\gamma$,
\begin{align}
\label{eq:i2}
  \E \sup_{s \in [0,t]} \| \Th_{\eps, \delta}(s) - \thl(s) \|^{p} \leq C(\| \Th_{\eps, \delta}(0) - \thl(0) \|  + \eps + \delta)^{\gamma} \to 0, \; \textrm{ as } \eps, \delta \to 0,
\end{align}
for a constant $C>0$, where $\gamma$ and $C$ are both independent of
$\eps, \delta$ and $\E$ denotes the expected value (statistical mean).
Furthermore we show that, for any $t > 0$,
\begin{align}
\label{eq:i3}
  \E \int_{0}^{t}\| (\U_{\eps, \delta}(s),\B_{\eps, \delta}(s)) - M(\thl)(s)\|^{2}ds  \leq C(\| \Th_{\eps, \delta}(0) - \thl(0) \|  + \eps + \delta)^{\gamma} \to 0,  \; \textrm{ as } \eps, \delta \to 0.
\end{align}
For the precise statement see Theorem \ref{thm:f:t:eod} below.
\item[(ii)]  Regarding the active scalar equation \eqref{eq:AS:th} we show that, provided
\begin{align}
	\text{$\alpha_{e_1, m}$, $\alpha_{e_2, m}$,
		and $\alpha_{e_3, m}$ are non-zero
		for each $m \in \{0, 1\}$,}
 \label{eq:s:n:config}
 \end{align}
 then \eqref{eq:AS:th} satisfies a form of the H\"ormander bracket condition. See \eqref{eq:Hormander:easy} and
 Theorem \ref{lem:hbc} below.\footnote{While we provide complete details
  only under the assumption \eqref{eq:s:n:config}, in principle many other noise configurations could be addressed with our methods.}
    Invoking recent results of Hairer and Mattingly \cite{HairerMattingly2008, HairerMattingly2011}
    (see also \cite{HairerMattingly06,FoldesGlattHoltzRichardsThomann2013})
    we infer that, in a suitably chosen Wasserstein metric $\WW$,
  \begin{align}
  \label{eq:cont}
     \WW( \mu^1P_t , \mu^2P_t ) \leq C e^{-\gamp t} \WW(\mu^1, \mu^2), \quad \text{ for } t \geq 0,
  \end{align}
  for any probability measures $\mu^1, \mu^2$, where the constants $C$, $\gamp$ are independent of $t\geq 0$
  and $\mu^1, \mu^2$.  Here $P_t$ is the Markov semigroup associated to \eqref{eq:AS:th}
  so that $\mu^iP_t $ represents the law of solutions to \eqref{eq:AS:th} at time $t \geq 0$ initially distributed as $\mu^i$.
  In particular \eqref{eq:cont} immediately implies that, under the given condition \eqref{eq:s:n:config}, the equation
  \eqref{eq:AS:th} has a unique statistically stationary state $\mu$.  See Theorem~\ref{thm:c:w}
  for further details.

\item[(iii)] Assume the noise satisfies \eqref{eq:s:n:config}.
We prove that, for every $\eps, \delta > 0$, the system \eqref{eq:MHD:vel}--\eqref{eq:MHD:temp} possesses at least one statistically stationary state $\mu_{\eps, \delta}$ which satisfies certain exponential moment bounds independently of $\eps, \delta$.
Moreover, we show that any such collection $\mu_{\eps, \delta}$ converges to $\mu$ in the limit as $\eps, \delta \to 0$, at an algebraic rate, in a suitable Wasserstein metric; see Theorem \ref{thm:SSS:conv} below for the precise statement.
In particular, this implies that for any sufficiently  regular observable $\phi$,
\begin{align}
\label{eq:conv:meas}
  \left| \int \phi(\U, \B, \Th) d \mu_{\eps, \delta} - \int \phi( M(\thl), \thl) d \mu  \right| \to 0, \;  \textrm{ in the limit as } \eps, \delta \to 0.
\end{align}
\end{itemize}

We now sketch our method of proof for the results described in (i)--(iii), highlighting some of the particular challenges we encountered.
Observe that the convergence to the formal limit \eqref{eq:AS:th} as $\eps,\delta \rightarrow 0$ is singular in the sense that
there is a phase space mismatch. It is therefore apparent that our analysis must involve multiple time scales:
an initial `layer', that is, short times depending on $\eps$ and $\delta$ influenced by the difference in initial conditions,
an intermediate time scale, and finally a large time scale where the loss of nonlinear terms in the momentum and magnetic
equations is especially apparent.   This long time scale is embodied in the statistically stationary states of \eqref{eq:MHD:vel}--\eqref{eq:MHD:temp}
and of \eqref{eq:AS:th} which are the primary focus of this work.

Our analysis builds on a method developed recently in \cite{FoldesGlattHoltzRichards2015}, which draws on a simple but powerful observation from \cite{HairerMattingly2008}: if one can establish a contraction property for the limit equation as in \eqref{eq:cont}, then
the question of convergence of statistically stationary states can be reduced to the convergence of solutions on finite time scales (cf. \eqref{eq:i2}--\eqref{eq:i3}).  
Nevertheless proving the results herein requires significant novel developments. 
This is due in part to the specific structural challenges inherent in  \eqref{eq:MHD:vel}--\eqref{eq:MHD:temp}.  However, 
in the course of our work here, we develop general techniques to tackle multi-parameter problems as well as methods for 
lifting the convergence of statistics to an extended phase space.

As a first step we establish the convergence of solutions on finite time scales without imposing
any relative rates for which $\eps, \delta$ vanish.  To the best of our knowledge, such convergence results are
new for the MHD system, even in the deterministic setting, where our methods apply with straight-forward modifications.
Here one of our key observations is that a difference in initial conditions for the velocity and magnetic components has
 a negligible effect on the temperature, namely of an algebraic order of $\eps + \delta$.
 We also use non-trivial cancelations in the nonlinear terms to take advantage
 of the improved regularity properties of the limit equation; cf. \eqref{eq:3a} in Section~\ref{sec:fta}.

Additional challenges arise due to the probabilistic nature of \eqref{eq:MHD:vel}--\eqref{eq:MHD:temp}
 and \eqref{eq:n:s}.  For example, we find that estimates on the difference of solutions to the temperature equation lead to an integral inequality
 that requires certain almost sure bounds.  This leads us
 to develop a probabilistic modification of the Gr\"onwall inequality which relies on a delicate stopping time argument.
 See \eqref{pfe} and Lemma~\ref{lem:prob:gron} below.  In particular this approach leads to algebraic rates of convergence in temperature,
 as in \eqref{eq:i2}.  We are then able to transfer these algebraic rates to the convergence of the velocity and magnetic fields, see \eqref{eq:i3}. 
 Note that, due to the existence of the initial layer, one should
not expect the convergence as $\eps, \delta \to 0$ to be uniform in the $\U$ and $\B$ components
up to the initial time.

In order to obtain more precise asymptotics, one could seek to extend classical approaches involving
multiple time scales to our setting; see for example \cite{OMalley1974,KabanovPergamenshchikov2003}.
Such an approach would yield a `corrector', or in our multi-parameter case a series of correctors, whose
structure depends on the relative sizes of $\eps$ and $\delta$.  These correctors provide well-posed effective dynamics which
accurately approximate %see Remark \alert{(FILL)} below),
the full system \eqref{eq:MHD:vel}--\eqref{eq:MHD:temp} up to time zero when $\eps$ and $\delta$ are small.   See Remark~\ref{rmk:cor:anal} below.
While the accuracy of these correctors can be rigorously
established with estimates similar to those carried out in detail in Section~\ref{sec:fta}, our approach, which is focused on the convergence of long time 
statistics,  does not require such `intermediate' systems, so we omit these details.
For results in a similar setting see \cite{FoldesGlattHoltzRichards2015} and e.g. \cite{Wang2004a} for an analogous situation in a
deterministic context.

Having established suitable finite time convergence results we turn to analyze the limit equation \eqref{eq:AS:th}
in pursuit of \eqref{eq:cont}. Here we underline that a proof of the estimate \eqref{eq:cont} becomes tractable by including a stochastic forcing.
The estimate \eqref{eq:cont} immediately implies that statistically invariant states
of \eqref{eq:AS:th} are unique and ergodic, providing a rigorous foundation for some basic assumptions in statistical theories of turbulence; cf. \cite{Frisch95} or \cite{moffatt2008} in our geophysical setting.
Such considerations from turbulence
motivate the study of unique ergodicity in a variety of infinite-dimensional stochastic systems arising in fluids; see e.g.
\cite{FlandoliMaslowski1,HairerMattingly06,BarbuDaPrato2007, Debussche2011a,KuksinShirikian12,
FoldesGlattHoltzRichardsThomann2013, ConstantinGlattHoltzVicol2013, FriedlanderGlattHoltzVicol2014}
and references therein.

Motivated by the fundamental postulates of turbulence where energy cascades from large to small spatial scales,
we consider \eqref{eq:AS:th} driven by a `spectrally degenerate' stochastic forcing, namely where the noise acts through a
narrow range of frequencies.
While this degenerate setting may be attractive physically, it makes the analysis leading to
\eqref{eq:cont} substantially more difficult compared to the case when noise is active on all spatial scales.
To tackle this situation we implement a strategy from the recent groundbreaking works of Hairer and Mattingly
\cite{HairerMattingly06, HairerMattingly2008, HairerMattingly2011}.  These works develop
a theory of hypo-ellipticity for infinite-dimensional systems which allows one to establish unique ergodicity and contractivity
for certain semilinear stochastic PDEs with spectrally degenerate stochastic forcing.
Application of this theory requires demonstrating that the stochastic excitation can propagate to higher frequencies through the nonlinearity.
 This is done by verifying that a H\"ormander bracket condition is satisfied by the set of activated modes.  Such an analysis requires a
 detailed understanding of the nonlinear structure in the governing equations, as observed in previous works
 \cite{EMattingly2001, Romito2004, HairerMattingly06, HairerMattingly2011, FoldesGlattHoltzRichardsThomann2013,FriedlanderGlattHoltzVicol2014}.
In our setting we rely on the explicit formulation of $M$ in \eqref{eq:smb:u1}--\eqref{eq:smb:u2}
and a non-trivial construction to avoid degeneracies in frequency space.

The final stage of our analysis demonstrates that
 the unique invariant statistics for the limit equation \eqref{eq:AS:th} approximate any reasonable invariant
statistics for the original system \eqref{eq:MHD:vel}--\eqref{eq:MHD:temp} for $\eps,\delta$ small.
Indeed, as we already
identified above, the bounds \eqref{eq:i2}--\eqref{eq:i3} and \eqref{eq:cont} provide a means to transfer the algebraic rates in \eqref{eq:i2}
to the convergence of invariant states in a suitable Wasserstein metric following \cite{HairerMattingly2008,FoldesGlattHoltzRichards2015}.
In particular, such Wasserstein metrics are closely related to the topology of weak convergence and hence our results imply
that observations from a reasonable class of invariant states for \eqref{eq:MHD:vel}--\eqref{eq:MHD:temp} converge to
observations relative to the unique invariant measure for \eqref{eq:AS:th} in the limit as $\eps,\delta\rightarrow 0$; see main result (iii) above.

It should be emphasized that in our previous work \cite{FoldesGlattHoltzRichards2015} on the stochastic Boussinesq system, 
convergence of statistically invariant states on the temperature component (in the infinite Prandtl number limit) was established by taking a similar approach.
In contrast, in this current work we prove convergence of invariant states on the extended phase space, including the velocity and magnetic fields, which 
requires us to overcome a novel challenge arising due to the phase space mismatch between the small parameter system \eqref{eq:MHD:vel}--\eqref{eq:MHD:temp}
 and the limit equation \eqref{eq:AS:th}.
This is accomplished by exploiting the structure of a Wasserstein distance and estimates on the constitutive law \eqref{eq:smb:u1}--\eqref{eq:smb:u2}
in order to 'lift' the contractive property for \eqref{eq:AS:th} to the extended phase space.
This strategy should prove useful for other systems with analagous singular limits, in particular for the Boussinesq system.

\subsection*{Organization of the Paper}

The rest of the  manuscript is organized as follows.
In Section \ref{sec:not} we establish our notational conventions 
and recall some mathematical foundations of the equations \eqref{eq:MHD:vel}--\eqref{eq:MHD:temp} and \eqref{eq:AS:th};
existence and uniqueness of solutions, and the associated Markovian framework.  In Section \ref{sec:fta} we prove the finite time convergence of solutions to
\eqref{eq:MHD:vel}--\eqref{eq:MHD:temp} to those of \eqref{eq:AS:th} in the limit as $\eps,\del \rightarrow 0$, as described in point (i) above.  Section \ref{sec:lim} contains
some background on Wasserstein metrics, provides a proof of the contractive property \eqref{eq:cont} as discussed in point (ii), and the
 uniqueness of statistically stationary states for \eqref{eq:AS:th} follows as a consequence.
To prove \eqref{eq:cont} we carry out the relevant infinite-dimensional Lie bracket analysis to demonstrate that a form of the H\"{o}rmander condition is satisfied.  The proof of weak
convergence of statistically stationary states, as described in point (iii) above, is given in Section \ref{sec:conv}.  Finally, we include an appendix with parameter independent exponential moment bounds which will be utilized throughout the manuscript.

\section{Preliminaries}
\label{sec:not}

We begin by recalling some details about the functional setting of \eqref{eq:MHD:vel}--\eqref{eq:MHD:temp} and the formal limit equation
\eqref{eq:AS:th}.

The main function spaces are
\begin{align*}
	H :=  \left\{U \in L^2(\TT^3)^3: \nabla\cdot U  = 0,\int_{\TT^3} U\,dx =0\right\}
	&\times \left\{B \in L^2(\TT^3)^3: \nabla\cdot B  = 0,\int_{\TT^3} B\,dx =0\right\} \\
\times  \left\{\Theta \in L^2(\TT^3):\int_{\TT^3} \Theta\,dx =0 \right\},
\end{align*}
and
\begin{align*}
	V := H \bigcap \left\{U \in H^1(\TT^3)^3\right\}
	\times \left\{B \in H^1(\TT^3)^3\right\} \times \left\{\Theta \in  H^1(\TT^3)\right\},
\end{align*}
where $H^s(\TT^3)$ are the usual Sobolev spaces.
For the limit problem \eqref{eq:AS:th} we take $H' =  L^2(\TT^3)$ and $V' = H^1(\TT^3)$ with the mean free condition imposed.
We denote by $\|\cdot\|$ and $\langle \cdot,\cdot \rangle$ the usual $L^2$-norm and inner product, respectively.   For any $p \geq 2$
  we set
  \begin{align*}
  \|\sigma\|_{L^p}^p := \int_{\TT^2} \biggl( \sum_{\substack{k \in \ZZ^3_0,\\ m = 0, 1}} |\alpha_{k, m} \sigma_k^m(x) |^2 \biggr)^{p/2} dx \,,
  \end{align*}
  so that $\| \sigma \|_{L^2} = \| \sigma \|$ reduces to the usual Hilbert-Schmidt norm.  Note that $\|\sigma\|$ is controlled by 
  $\|\sigma\|_{L^p}$ for every $p\geq 2$.

We have the following existence results for \eqref{eq:MHD:vel}--\eqref{eq:MHD:temp}:

\begin{Prop}[Martingale solutions for the full system]\label{prop:weak:sol:full}
Consider \eqref{eq:MHD:vel}--\eqref{eq:MHD:temp} with any $\eps, \delta > 0$ and where
\eqref{eq:n:s} holds with only finitely many $\alpha_{k,m} \not = 0$ (or where the $\alpha_{k,m}$ decay
sufficiently rapidly as $|k| \to \infty$),
\begin{itemize}
  \item[(i)] Given any initial probability distribution $\mu \in Pr(H)$ there exists a
  stochastic basis $\mathcal{S} = (\Omega, \mathcal{F}, \Prb, \{\mathcal{F}_t\}_{t \geq 0}, W)$\footnote{Recall that
  a stochastic basis consists of a probability space, a right-continuous, complete filtration $\{\mathcal{F}_t\}_{t \geq 0}$
  along with a collection $W = \{ W^{k, m} \}$ of independent,
  identically distributed $1d$ Brownian motions adapted to this filtration.}
  and a corresponding
  stochastic process $(\U, \B, \Th)$ with
  \begin{align*}
	(\U, \B, \Th) \in L^2(\Omega; L^\infty_{loc}([0,\infty), H) \cap L^2_{loc}([0, \infty), V)),
  \end{align*}
  which is weakly continuous, adapted to the filtration provided by the stochastic basis,
  (weakly) solves \eqref{eq:MHD:vel}--\eqref{eq:MHD:temp}, and such that \emph{Law}$(\U(0), \B(0), \Th(0)) = \mu$.
  We say that $(\mathcal{S}, (\U, \B, \Th))$ is a \emph{Martingale solution of \eqref{eq:MHD:vel}--\eqref{eq:MHD:temp}}.
  \item[(ii)]  Moreover, for every $\eps,\delta > 0$,
  there exists a stationary Martingale solution $(\mathcal{S}, (\U_S, \B_S, \Th_S))$ such that this family of solutions satisfies the uniform moment bound
  \begin{align}
     \sup_{\eps,\delta \in(0,N]} \E  \exp(\eta (\eps \|\U_S\|^2 + \delta \|\B_S\|^2 + \| \Th_S\|^2_{L^3})) \leq  C_N < \infty,
     \label{eq:un:wk:stat:bnd}
  \end{align}
  for every $N>0$, where $\eta, C_N > 0$ are independent of $\eps, \delta > 0$.
  \end{itemize}
\end{Prop}

\noindent Proposition~\ref{prop:weak:sol:full}, (i) can be established in a similar manner as for the 3D Navier stokes equations,
namely through a Galerkin regularization procedure, see e.g. \cite{AlbeverioFlandoliSinai2008}.  Regarding the
existence of stationary Martingale solutions, Proposition~\ref{prop:weak:sol:full}, (ii),
the proof proceeds exactly as in \cite{FoldesGlattHoltzRichardsWhitehead2015} using Lemmas~\ref{lem:mom:bnds:1}, \ref{l:imb}
in the Appendix below.   Here again the existence of stationary solutions follows as a limit of invariant states of a regularization of the governing equations
following \cite{AlbeverioFlandoliSinai2008}, but with an additional twist.
The idea is to consider a Galerkin truncation only in the velocity and magnetic components of
\eqref{eq:MHD:vel}--\eqref{eq:MHD:temp} (cf. \cite{FoldesGlattHoltzRichardsWhitehead2015}).
This allows the regularization procedure to preserve the advection-diffusion structure in \eqref{eq:MHD:temp} which
in turn permits the $\eps, \delta$-uniform bound on $L^3(\TT^3)$ norms of $\Th$ in \eqref{eq:un:wk:stat:bnd}
(or for that matter on $L^p(\TT^3)$ for any $p\geq 2$).    This uniform bound is needed for
Theorems~\ref{thm:f:t:eod} and \ref{thm:SSS:conv} below.

The formal limit equation is much more tractable analytically, and possesses unique, pathwise solutions.
\begin{Prop}[Well-Posedness for the limit equation]\label{prop:well:pose:lim}  Consider \eqref{eq:AS:th}
supplemented with \eqref{eq:AS:vel}, \eqref{eq:AS:mag} and subject to the same condition on $\sigma$ given by \eqref{eq:n:s}
as in Proposition \ref{prop:weak:sol:full}.
\begin{itemize}
  \item[(i)] Fix a stochastic basis $\mathcal{S}$ and any initial condition $\thl_0 \in L^2(\Omega; H')$
  which is $\mathcal{F}_0$ measurable.  Then there exists a unique
  \begin{align*}
	\thl \in L^2(\Omega; C([0,\infty), H') \cap L^2_{loc}([0, \infty), V'))
  \end{align*}
  which is $\mathcal{F}_t$-adapted, solves \eqref{eq:AS:th} and satisfies the initial condition $\thl(0) = \thl_0$.   We say that
  $\theta$ is a \emph{Pathwise solution of  \eqref{eq:AS:th}}.
  \item[(ii)] Writing
  $\theta(t, \theta_0)$ for the solution with initial condition $\theta_0 \in H'$ at time $t \geq 0$, we have for any $t \geq 0$
  \begin{align*}
     \theta(t, \theta_0^n) \to \theta(t, \theta_0) \; \; a.s. \textrm{ in } H' \quad \textrm{whenever}  \quad \theta_0^n \to \theta_0 \textrm{ in } H'.
  \end{align*}
  \item[(iii)] Pathwise solutions $\theta$ of \eqref{eq:AS:th} satisfy the exponential moment bounds \eqref{eq:mod:ad:1}--\eqref{eq:mod:ad:3}.
  \end{itemize}
\end{Prop}
\noindent This well-posedness result for \eqref{eq:AS:th} is again a straightforward
extension of known results in the deterministic case, cf. \cite{FriedlanderSuen2015},  particularly since we
are working with an additive noise.  For discussion of the exponential moment bounds, see Lemma \ref{lem:mom:bnds:1} in the Appendix.

We finally recall the Markovian setting of \eqref{eq:AS:th} as follows.
Denote by $\text{Pr}(H')$ the space of Borel probability measures on $H'$.
Also denote by $M_b(H')$ and $C_b(H')$ the
bounded measurable and bounded continuous (real valued) functions on $H'$, respectively.  The Markov transition functions
are defined by
\begin{align*}
   P_t( \theta_0, A) = \Prb( \theta(t, \theta_0) \in A) \quad \textrm{ for any } t \geq 0, \theta_0 \in H', A \in \mathcal{B}(H'),
\end{align*}
where $\mathcal{B}(H')$ is the set of Borel subsets of $H'$
and the associated semigroup $\{P_{t}\}_{t \geq 0}$ acts on `observables' $\phi \in M_b(H')$ according to
\begin{align*}
  P_t \phi(\theta_0) :=  \int_{H'}  \phi(\theta) P_t( \theta_0, d\theta) = \E \phi(\theta(t, \theta_0)),
\end{align*}
and on `initial distributions' $\mu \in Pr(H')$ by
\begin{align*}
  \mu P_t (A) := \int_{H'}  P_t( \theta, A) d\mu(\theta).
\end{align*}
Recall that $\mu\in\Pr (H')$ is an invariant measure for $P_t$, i.e. an invariant measure for \eqref{eq:AS:th}, if
$\mu P_t = \mu$ for all $t\geq 0$.  Since $\theta_0 \mapsto \theta(t, \theta_0)$ is continuous by Proposition \ref{prop:well:pose:lim}
for each $t \geq 0$, we have that $P_t$ is \emph{Feller}, mapping $C_b(H')$ to itself.

Recall that the Krylov-Bogolyubov procedure and the compact embedding of $H^1$ into $L^2$ immediately produce the existence
of an invariant measure for \eqref{eq:AS:th}--\eqref{eq:AS:mag}, while the uniqueness, a much more delicate issue, is addressed below
in Section~\ref{sec:lim}.  Using the moment bound \eqref{eq:mod:ad:3} we infer that, for any invariant measure and any $p \geq 2$
\begin{align}
   \int  \| \theta \|_{L^p} d \mu(\theta)  \leq C < \infty,
   \label{eq:stat:exp:lim}
\end{align}
where the constant $C$ is as in \eqref{eq:un:wk:stat:bnd} above.

\section{Finite Time Convergence Analysis}
\label{sec:fta}

In this section we establish the finite time convergence of solutions of \eqref{eq:MHD:vel}--\eqref{eq:MHD:temp} to solutions
of \eqref{eq:AS:th} in the limit as $\eps, \delta \to 0$, see (i) in the summary of our main results.
We again emphasize that no relative rates between $\eps$ and $\delta$ are required here.
In Section~\ref{sec:conv} these results will be used to establish the convergence of stationary solutions (see Theorem~\ref{thm:SSS:conv} below).  The section concludes with some remarks concerning `corrector systems'
which provide a different well-posed approximation to \eqref{eq:MHD:vel}--\eqref{eq:MHD:temp}, namely an approximation that is accurate up to time zero.

Note that we are working only with zero mean functions and vector fields.  This restriction justifies the use of the Poincar\' e inequality
which we apply below without further comment.

\begin{Thm}\label{thm:f:t:eod}
For every $\eps, \delta \in (0, 1]$ let $(\U, \B, \Th)$ be a Martingale solution
of \eqref{eq:MHD:vel}--\eqref{eq:MHD:temp} in the sense of Proposition~\ref{prop:weak:sol:full} and let
$\thl$ be a solution of
\eqref{eq:AS:th} as in Proposition~\ref{prop:well:pose:lim}.
Suppose that there is $\eta_0 > 0$ such that
\begin{align}
   \sup_{\eps, \delta \in (0, 1]}
   \E  \exp(\eta_0 (\eps \|\U(0)\|^2 + \delta \|\B(0)\|^2 + \| \Th(0)\|^2_{L^3} + \| \thl(0)\|^2_{L^3})) \leq C_0 < \infty.
   \label{eq:ft:d:a:1}
\end{align}
Then there exists $\eta_1 = \eta_1 (\eta_0, \nu, \kappa, C_0)$ such that for each  $T > 0$, $\eta \in (0, \eta_1]$, $\gamma \in (0, \eta/(\eta + CT)]$,
and $p>\gamma$, there is a constant
$C = C(p,\eta_0, \nu, \kappa, C_0)$ such that
\begin{align}
\label{eq:conv:th:f:t}
   \E &\sup_{t \in [0,T]} \|  \Th(t) - \thl(t) \|^{p}
   \\ &\leq C \exp(\eta T \|\sigma\|_{L^3}^2) \left((\eps + \delta)^{\gamma} +
   \left(\E \| \Th(0) - \thl(0)\|^2 + \eps \|\U(0) - M_\ul(\thl)(0)  \|^2
   + \delta  \| \B(0) - M_\bl(\thl)(0) \|^2 \right)^{\gamma} \right)^{1/2}, \notag
\end{align}
and
\begin{align}
 \label{eq:conv:ub:f:t}
   &\E \int_0^T  \left(\|\U(t) - M_\ul(\thl)(t)\|^{2}_{H^1} +  \|B(t) - M_\bl(\thl)(t)\|^{2}_{H^1} \right) dt \\
    &\quad \leq C T \exp(\eta T \|\sigma\|_{L^3}^2) \left((\eps + \delta)^{\gamma} +
   \left(\E \| \Th(0) - \thl(0)\|^2  + \eps \|\U(0) - M_\ul(\thl)(0)  \|^2
   + \delta  \| \B(0) - M_\bl(\thl)(0) \|^2\right)^{\gamma} \right)^{1/2} \! \! \!. \notag
\end{align}
\end{Thm}

\begin{Rmk}
As identified in Proposition \ref{prop:weak:sol:full}, \eqref{eq:ft:d:a:1} holds for a collection of stationary Martingale
solutions of \eqref{eq:MHD:vel}--\eqref{eq:MHD:temp}.  The bound \eqref{eq:ft:d:a:1} holds for the unique stationary state of
\eqref{eq:AS:th} due to \eqref{eq:stat:exp:lim}.   These observations will play a crucial role below in Section~\ref{sec:conv}.
\end{Rmk}

Before proceeding to the proof we set notations for several operators that are used extensively.
We denote the Coriolis-Stokes operator by $A$ which acts on
sufficiently smooth, divergence free vector fields according to
\begin{align}
 \OA \ul  := -\nu \Delta \ul + P \Omm \times \ul \,,
  \label{c0}
\end{align}
where $P$ is the Leray projector onto divergence free vector fields.
Note that $P$ commutes with $\Delta$, and therefore we omit
$P$ in front of the Laplacian term in \eqref{c0}. Since $u \mapsto
\Omm \times u$ is a skew symmetric operator, then $\langle P \Omm \times \ul, \ul \rangle = \langle \Omm \times \ul, \ul \rangle = 0$,
and therefore
$A$ is positive definite on zero mean functions with
\begin{align}
	\langle A u, u \rangle = \nu \|\nabla u\|^2 \,,
	\quad  \textrm{ for all } u \in H^1(\TT^3)^3 \textrm{ with } \nabla\cdot u  = 0
	\label{eq:A:pos:def}
\end{align}
which we use below without further comment.
We also make use of the operator $Q$ defined as
\begin{align}
  \OQ  \ul  := A  \ul - (-\Delta)^{-1} (\BBB\cdot \nabla)^2  \ul \,,
  \label{eq:oq:op}
\end{align}
where $(-\Delta)^{-1}$ has the range of zero mean functions.
Note that
\begin{align*}
\langle Q u, u \rangle \geq \langle A u, u \rangle =
\nu \|\nabla u\|^2, \,  \quad  \textrm{ for all } u \in H^1(\TT^3)^3 \textrm{ with } \nabla\cdot u  = 0,
\end{align*}
so that, in particular, $Q$ is positive definite.
Thus $A$ and $Q$ are invertible operators, and we may
rewrite the limit equation
\eqref{eq:AS:vel}--\eqref{eq:AS:mag} as
\begin{equation}\label{cnt}
	\ul = Q^{-1} P \bfe_3 \thl \,, \qquad \bl = (-\Delta)^{-1} (\BBB \cdot \nabla \ul)\,.
\end{equation}
\begin{Rmk}
  Notice that by \eqref{cnt}, $\ul$ and $\bl$ are respectively two
  and three degrees smoother than $\theta$. By this we mean that for any $s \in \RR$ if $\theta \in H^s$, then $\ul  \in H^{s+2}$ and
  $\bl \in H^{s +3}$, as can be seen from corresponding symbols given explicitly in \eqref{eq:smb:u1}.
\end{Rmk}

\begin{proof}[Proof of Theorem \ref{thm:f:t:eod}]
In what follows $C$ denotes any constant which depends on $\kappa$, $\nu$, and
$C_0$, but which is independent of $\eps$, $\delta$, $K$, $\sigma$, and $t>0$.

Working from \eqref{cnt}, \eqref{eq:AS:th} we have
\begin{equation}\label{ilu}
d\ul = Q^{-1} P \bfe_3 d\thl =  Q^{-1} P \bfe_3 (\kappa \Delta \thl
- \ul \cdot \nabla \thl) dt + Q^{-1} P \bfe_3 \sigma dW \,.
\end{equation}
Observe \eqref{eq:AS:vel} takes the form
$A\ul = P \BBB \cdot \nabla\bl + P \bfe_3 \thl$ and adding it to
an $\eps$ multiple of \eqref{ilu} we obtain
\begin{equation} \label{fcu}
\eps du + Au dt = (P \BBB \cdot \nabla b + P \bfe_3 \thl) dt +
\eps Q^{-1} P \bfe_3 (\kappa \Delta \thl
- \ul \cdot \nabla \thl) dt + \eps Q^{-1} P \bfe_3 \sigma dW \,.
\end{equation}
Next, we derive an evolution equation for $b$. From
\eqref{cnt}, \eqref{eq:AS:th} we obtain that
\begin{equation}\label{ilb}
db = (-\Delta)^{-1} \BBB \cdot \nabla du = R (\kappa \Delta \thl
- \ul \cdot \nabla \thl) dt + R \sigma dW \,,
\end{equation}
where
\begin{equation}\label{eq:R}
R := (-\Delta)^{-1} (\BBB \cdot \nabla) Q^{-1} P \bfe_3 \,.
\end{equation}
Adding a $\delta$-multiple of \eqref{ilb} to \eqref{eq:AS:mag}
yields
\begin{equation}\label{fcb}
\delta d\bl - \Delta \bl dt = (\BBB \cdot \nabla \ul) dt +
\delta R (\kappa \Delta \thl
- \ul \cdot \nabla \thl) dt + \delta R \sigma dW \,.
\end{equation}
Set
\begin{align}\label{eq:not}
	v :=  \U - \ul, \quad
	f := \B - \bl , \quad \phi := \Th -  \thl,
\end{align}
where $(\U, \B, \Th)$ is solution of
\eqref{eq:MHD:vel}--\eqref{eq:MHD:temp}.
Comparing \eqref{fcu} and \eqref{eq:MHD:vel} we obtain
\begin{equation}\label{d:vel}
\eps dv + Av dt = (P \BBB\cdot \nabla f + P \bfe_3 \phi)dt + P(\delta \B \cdot \nabla \B - \eps \U \cdot \nabla \U)dt - \eps Q^{-1} P \bfe_3 (\kappa \Delta \thl
- \ul \cdot \nabla \thl)dt - \eps Q^{-1} P \bfe_3 \sigma dW \,.
\end{equation}
Similarly by substracting \eqref{fcb} from \eqref{eq:MHD:mag} we have
\begin{equation}\label{d:mag}
	\delta df  - \Delta f dt
		= \BBB \cdot \nabla v dt -
		\delta (\U\cdot \nabla \B - \B \cdot \nabla \U) dt
		-\delta R (\kappa \Delta \thl - \ul \cdot \nabla \thl) dt - \delta R \sigma dW.
\end{equation}
Next, comparing \eqref{eq:MHD:temp} and \eqref{eq:AS:th} we obtain
\begin{equation}\label{d:tem}
\partial_t \phi - \kappa \Delta \phi = - \U \cdot \nabla \phi - v \cdot \nabla \thl \,.
\end{equation}
Applying the It\= o formula to \eqref{d:vel}
and recalling \eqref{eq:A:pos:def} yields
\begin{multline}\label{if:vel}
\frac{\eps}{2} d \|v\|^2 + \nu \|\nabla v\|^2 dt =
\left(\langle \BBB \cdot \nabla f, v \rangle +
 \langle \bfe_3 \phi, v \rangle  \right) dt
- \eps \langle \U \cdot \nabla \U, v \rangle dt
+ \delta \langle \B \cdot \nabla \B, v \rangle dt
\\
- \eps \langle   Q^{-1} P \bfe_3 (\kappa \Delta \thl
- \ul \cdot \nabla \thl), v \rangle dt
+ \frac{\eps}{2} \|Q^{-1} P \bfe_3 \sigma \|^2 dt
 -
\eps \langle Q^{-1} P \bfe_3 \sigma, v \rangle dW \,.
\end{multline}
Similarly \eqref{d:mag} yields
\begin{align}
\frac{\delta}{2} d\|f\|^2 +  \|\nabla f\|^2 dt =
 	\langle &\BBB \cdot \nabla v, f\rangle dt
	- \delta \langle \U \cdot \nabla \B - \B \cdot \nabla \U, f \rangle dt
	- \delta \langle R (\kappa \Delta \thl - \ul \cdot\nabla \thl), f \rangle dt
	\notag\\
	&+ \frac{\delta}{2} \|R\sigma\|^2 dt - \delta \langle R\sigma, f \rangle dW \,.
	\label{if:mag}
\end{align}
Finally from \eqref{d:tem} and the fact that $\nabla \cdot v = \nabla \cdot U = 0$,
\begin{equation}\label{if:tem}
\frac{1}{2} \frac{d}{dt}\|\phi\|^2 + \kappa \|\nabla \phi\|^2 =
\langle v \cdot \nabla \phi, \thl \rangle \,.
\end{equation}
Addition of \eqref{if:vel} and \eqref{if:mag}, and integration in time up to $t \wedge \tau$, where $\tau$ is any stopping time,  gives
\begin{align*}
\frac{1}{2} ( \eps \|v(t \wedge \tau)\|^2
&+ \delta \|f(t \wedge \tau)\|^2) +
	\int_0^{t \wedge \tau}( \nu \|\nabla v\|^2 + \|\nabla f\|^2) \, ds
	=
	\frac{1}{2} ( \eps \|v(0)\|^2 + \delta \|f(0)\|^2)\\
	&+ \int_0^{t \wedge \tau} \langle \bfe_3 \phi, v \rangle\,  ds
	- \eps \int_0^{t \wedge \tau} \langle \U \cdot \nabla \U, v \rangle \, ds
	+ \delta \int_0^{t \wedge \tau} \left( \langle \B \cdot \nabla \B, v \rangle
	- \langle \U \cdot \nabla \B - \B \cdot \nabla \U, f \rangle \right)\, ds
\\
	&+ \eps \int_0^{t \wedge \tau} \left(
	\frac{1}{2} \|Q^{-1} P \bfe_3 \sigma \|^2 -
	\langle   Q^{-1} P \bfe_3 (\kappa \Delta \thl
	- \ul \cdot \nabla \thl), v \rangle  \right) \, ds
\\
&+ \delta \int_0^{t \wedge \tau}
  \left( \frac{1}{2} \|R\sigma\|^2 -  \langle R (\kappa \Delta \thl - \ul \cdot\nabla \thl), f \rangle \right) \, ds
- \int_0^{t \wedge \tau}  \left( \eps \langle Q^{-1} P \bfe_3 \sigma, v \rangle  + \delta \langle R\sigma, f \rangle \right) dW \,.
\end{align*}
Here we note that $W$ is the same process in both \eqref{if:vel}, \eqref{if:mag}, and we have used that $(\BBB\cdot \nabla)$
is anti-symmetric to cancel terms.
Next, we estimate terms on the right hand side. First,
\begin{align*}
|\langle  \bfe_3 \phi, v \rangle | \leq \|\phi\| \|v\| \leq
\frac{\nu}{8} \|\nabla v\|^2 + \frac{C}{\nu} \|\phi\|^2 \,.
\end{align*}
Since $\langle U \cdot \nabla v, v \rangle = 0$,
$\ul$ is two degrees smoother than $\thl$ (by \eqref{cnt}),
and using the embedding $H^2 \hookrightarrow L^{\infty}$, we have
\begin{align}
|\langle \U \cdot \nabla \U, v \rangle| &=
|\langle \U \cdot \nabla \ul, v \rangle| = |\langle \U \cdot \nabla v, \ul \rangle|
\leq \|\U\| \|\nabla v\| \|\ul\|_{L^\infty} \leq
C\|\U\| \|\nabla v\| \|\ul\|_{H^2} \leq
C\|\U\| \|\nabla v\| \|\thl\|
\notag \\
&\leq \frac{\nu}{8\eps} \|\nabla v\|^2 + \frac{C\eps}{\nu}
\|\U\|^2  \|\thl\|^2 \,.
\label{eq:3a}
\end{align}
Next observe, again using \eqref{eq:not} and that all of the vector field are divergence free,
\begin{align*}
\langle \B \cdot \nabla \B, v \rangle
- \langle \U\cdot \nabla \B - \B \cdot \nabla \U, f \rangle &=
\langle \B \cdot \nabla f, v \rangle + \langle \B \cdot \nabla \bl, v \rangle
- \langle \U \cdot \nabla \bl - \B \cdot \nabla v - \B \cdot \nabla \ul, f \rangle
\\
&=
 \langle \B \cdot \nabla \bl, v \rangle
- \langle \U \cdot \nabla \bl - \B \cdot \nabla \ul, f \rangle\\
&=
-\langle \B \cdot \nabla v, \bl \rangle
+ \langle \U \cdot \nabla f, \bl \rangle - \langle \B \cdot \nabla f, \ul \rangle
\,,
\end{align*}
and consequently, using again that both $\ul$ and $\bl$ are two degrees smoother than $\thl$,
similarly to \eqref{eq:3a}, we have
\begin{align*}
|\langle \B \cdot \nabla \B, v \rangle
- \langle \U \cdot \nabla \B - \B \cdot \nabla \U, f \rangle|
	&\leq C( \|\B\| \|\nabla v\| \|\bl\|_{H^2} +
		\|\U\| \|\nabla f\| \|\bl\|_{H^2} +
		\|\B\| \|\nabla f\| \|\ul\|_{H^2}) \\
	&\leq C(\|\B\| \|\nabla v\| \|\thl\| +
		\|\U\| \|\nabla f\| \|\thl\| +
		\|\B\| \|\nabla f\| \|\thl\|)
\\
&\leq
\frac{1}{8\delta}  (\nu\|\nabla v\|^2 + \|\nabla f\|^2) +
\frac{C\delta}{\nu \wedge 1} \|\thl\|^2 (\|\U\|^2 + \|\B\|^2 )\,.
\end{align*}
Next observe that the range of $Q^{-1}$ consists of divergence free vector fields, and consequently
\begin{align*}
\langle Q^{-1} P \bfe_3 ( \ul \cdot \nabla \thl), v \rangle = \langle   \ul \cdot \nabla \thl,  \bfe_3 \cdot  Q^{-1} v \rangle = - \langle   \ul \cdot \nabla (\bfe_3 \cdot  Q^{-1} v), \thl \rangle
\end{align*}
Recalling that $Q^{-1}$ provides two degrees of smoothing, and using again $H^2 \hookrightarrow L^\infty$, we thus obtain
\begin{align*}
|\langle Q^{-1} P \bfe_3 (\kappa \Delta \thl - \ul \cdot \nabla \thl), v \rangle|
\leq C \|\theta\| \|v\| + C \|v\| \|\theta\| \|\ul\|_{H^2}
\leq \frac{\nu}{8\eps}  \|\nabla v\|^2  +
\frac{C\eps}{\nu} (\|\thl\|^2 + \|\thl\|^4) \,.
\end{align*}
Analogously since $R$ has also two (in fact three, cf. \eqref{eq:R}) degrees of smoothing,
\begin{align*}
|\langle R (\kappa \Delta \thl
- \ul \cdot \nabla \thl), f \rangle| \leq C\|\thl\| \|f\|  +
C \|f\| \|\thl\|\|u\|_{H^2}
\leq \frac{1}{8\delta}  \|\nabla f\|^2  +
C\delta(\|\theta\|^2 + \|\theta\|^4) \,.
\end{align*}
Finally, we trivially have
\begin{align*}
\|R\sigma\|^2 \leq C \|\sigma\|^2, \qquad \|Q^{-1} P \bfe_3 \sigma \|^2 \leq C \|\sigma\|^2 \,.
\end{align*}
Combining all of the preceding estimates we infer
\begin{multline}
\frac{1}{2}\int_0^{t \wedge \tau} (\nu \|\nabla v\|^2 + \|\nabla f\|^2) \, ds
\leq
\frac{1}{2} ( \eps \|v(0)\|^2
+ \delta \|f(0)\|^2) +
\frac{C}{\nu}\int_0^{t \wedge \tau} \|\phi\|^2 \, ds \\
+
C (\eps^2 + \delta^2) \left(1 + \frac{1}{\nu}\right) \int_0^{t \wedge \tau} \|\thl\|^2
(\|\U\|^2 + \|\B\|^2 + \|\thl\|^2 + 1) \, ds
+
 \frac{t \wedge \tau}{2} \|\sigma\|^2 (\epsilon + \delta) + M_{t\wedge\tau} \,,\label{fes}
\end{multline}
where $M_t$ represents a Martingale term.
Working from \eqref{if:tem} we have
\begin{align*}
\frac{1}{2} \frac{d}{dt}\|\phi\|^2 + \kappa \|\nabla \phi\|^2 \leq
\|v\|_{L^6} \|\nabla \phi\| \|\thl\|_{L^3} \leq
\kappa \|\nabla \phi\|^2 + \frac{C}{\kappa} \|\nabla v\|^2 \|\thl\|_{L^3}^2 \,,
\end{align*}
and therefore
\begin{align}\label{pfe}
\sup_{s \in [0, t \wedge \tau]} \|\phi(s)\|^2
\leq  \| \phi(0) \|^2 + \frac{C}{\kappa} \Big(\sup_{s \in [0, t \wedge \tau]} \|\thl(s)\|^{2}_{L^3}\Big)
\int_{0}^{t \wedge \tau} \|\nabla v\|^2 \, ds \,.
\end{align}
We will now implement a probabilistic modification of the Gr\"{o}nwall inequality.

\begin{Lem}
\label{lem:prob:gron}
Assume that an $\mathcal{F}_{t}$-adapted stochastic process $(\mcx (t))_{t \geq 0}$ satisfies, for each $K>0$,
\begin{equation}\label{jpgnw}
\E \mcx (t \wedge \tau_K) \leq K \mct + C K\int_0^t \E \mcx(s\wedge \tau_K) \, ds +
K^2 C (\eps + \delta)(1 + \|\sigma\|^2 t)  \,,
\end{equation}
where $\mct,\eps,\delta>0$ are constants, and $\{\tau_K\}$ is a collection of stopping times which satisfies
\begin{align*}
\Prb(\tau_K \leq t) \leq C' e^{-\eta K},
\end{align*}
for some $\eta > 0$. Then for any $t>0$ and $\gamma \leq \frac{\eta}{\eta+Ct}$ one has
\begin{equation}
\label{eq:gron:conc}
\E (\mcx (t))^{\gamma} \leq C_1 (\mct^\gamma + (\eps + \delta)^{\gamma}) \,,
\end{equation}
where $C_1=C_1(C',t,\|\sigma\|,\gamma)>0$.
\end{Lem}

We postpone the proof of Lemma \ref{lem:prob:gron} to the end of this section.
To implement Lemma \ref{lem:prob:gron}, we define the stopping times
\begin{align}
\label{eq:stop:def}
\tau_K := \inf_{s \geq 0} \{\|\thl(s)\|^2_{L^3} \geq K \}\,,
\end{align}
for any $K > 0$.
Note that for each $t>0$, and sufficiently small
$0<\eta \leq \eta_1(\kappa,\|\sigma\|_{L^3}, \eta_0)$,
one finds by the Markov inequality, \eqref{eq:mod:ad:1} and
the assumption \eqref{eq:ft:d:a:1},
\begin{align*}
\Prb(\tau_K \leq t) = \Prb(\sup_{s\in[0,t]}\|\thl(s)\|^2_{L^3}\geq K)
\leq e^{-\eta K}\E\exp\Big(\eta \sup_{s\in[0,t]}\|\thl(s)\|^2_{L^3}\Big)
\leq  C e^{-\eta K}\exp(\eta t \|\sigma\|_{L^3}^2).
\end{align*}
With $\tau = \tau_K$, standard energy estimates, \eqref{eq:mod:ad:2}, \eqref{avebuth},  and the assumption \eqref{eq:ft:d:a:1},
we can estimate a term on the right hand side of \eqref{fes} as
\begin{align}
 \E \int_0^{t\wedge \tau_K} \|\thl\|^2 &\left( \|\U\|^2 + \|\B\|^2  + \| \thl \|^2 \right) ds \leq
C K \E \int_0^t \left(\|\nabla \U\|^2 + \|\nabla \B\|^2  + \| \thl \|^2 \right) ds
\notag \\
&\leq
C K
 \E\left(\frac{\eps}{2} \| \U(0)\|^2 + \frac{\delta}{2} \|\B(0) \|^2 + \frac{C}{\kappa \nu}\|\Th(0)\|^{2} + \frac{1}{2\kappa}\| \thl(0)\|^2\right)
 +  CK\|\sigma\|^{2}t
 \notag \\
&\leq
 CK(1 + \|\sigma\|^{2}t) \,.
  \label{eq:gw:est:2}
\end{align}
Hence, by combining \eqref{fes} with \eqref{pfe} and using \eqref{eq:gw:est:2} we find for each $K \geq 1$,
the estimate \eqref{jpgnw} is satisfied, where
\begin{align}
\label{eq:chi}
\mcx (t) := \sup_{s \in [0, t]} \|\phi(s)\|^2\,,
\qquad
\mct := \E\left(\| \phi(0) \|^2 + C ( \eps \|v(0)\|^2 + \delta \|f(0)\|^2)\right) \,.
\end{align}
By Lemma \ref{lem:prob:gron} the estimate \eqref{eq:gron:conc} is satisfied with such $\mcx (t)$ and $\mct$.

We can now prove \eqref{eq:conv:th:f:t} by combining \eqref{eq:gron:conc}, \eqref{eq:mod:ad:2}, Corollary \ref{c:emb}, and the assumption \eqref{eq:ft:d:a:1}.  For any $p>\gamma$, we find
\begin{align}
\label{eq:con4}
\E\sup_{s\in[0,t]}&\|\phi(s)\|^{p} \leq C\E\sup_{s\in[0,t]}\|\phi(s)\|^{\gamma}(\|\Th(s)\|^{p-\gamma} + \|\thl(s)\|^{p-\gamma}) \notag \\
&\leq
C\Big(\E\big(\sup_{s\in[0,t]}\|\phi(s)\|^{2\gamma}\big)\Big)^{1/2}\left[\Big(\E \sup_{s\in[0,t]}\|\Th(s)\|^{2 (p-\gamma) }\Big)^{1/2}
+ \Big(\E \sup_{s\in[0,t]}\|\thl(s)\|^{2(p-\gamma)}
\Big)^{1/2}\right] \notag \\
&\leq C\Big(\E\big(\sup_{s\in[0,t]}\|\phi(s)\|^{2\gamma}\big)\Big)^{1/2}\left[\Big(\E\exp\big(\eta\sup_{s\in[0,t]}\|\Th(s)\|^{2}\big)\Big)^{1/2}
+ \Big(\E\exp\big(\eta\sup_{s\in[0,t]}\|\thl(s)\|^{2}\big)\Big)^{1/2}\right] \notag \\
&\leq C\exp\big(\eta t\|\sigma\|_{L^3}^{2}\big)
\Big[\mct^{\gamma} + (\eps + \delta)^{\gamma} \Big]^{1/2} \,.
\end{align}

Having established \eqref{eq:conv:th:f:t}, it remains to prove \eqref{eq:conv:ub:f:t}.
Note that from \eqref{fes} we have
\begin{align}
\label{eq:con3}
\E\int_{0}^{t}(\nu\|\nabla v\|^2 + &\|\nabla f\|^{2})ds \leq \E(\eps\|v(0)\|^{2} + \delta\|f(0)\|^{2}) + Ct\E\sup_{s\in[0,t]}\|\phi(s)\|^{2}  \notag \\
& + C(\eps^2 + \delta^2)\Big(\E\big(\sup_{s\in[0,t]}\|\thl(s)\|^2\big)^{2}\Big)^{1/2}
\left(\E\Big(\int_{0}^{t}(\|\U\|^2 + \|\B\|^2 + \|\thl\|^2)ds\Big)^{2}\right)^{1/2} \notag \\
& + C(\eps^2 + \delta^2)\E\int_{0}^{t}\|\thl\|^2 ds + (\eps + \delta)C\|\sigma\|^{2}t.
\end{align}

Combining \eqref{eq:conv:th:f:t} and \eqref{eq:con3}, and once more invoking \eqref{eq:mod:ad:2}, Corollary \ref{c:emb}, and \eqref{eq:ft:d:a:1}, we find
\begin{align}
\label{eq:con5}
\E\int_{0}^{t}(\nu\|\nabla v\|^2 + \|\nabla f\|^{2})ds &\leq \E(\eps\|v(0)\|^{2} + \delta\|f(0)\|^{2})
 + Ct \exp\big(\eta t\|\sigma\|_{L^3}^{2}\big)
\left[ \mct^\gamma +  (\eps + \delta)^{\gamma}\right]^{1/2}
\notag \\
& + C(\eps^2 + \delta^2)\exp\big(\eta t\|\sigma\|^{2}\big) + C(\eps + \delta)\|\sigma\|^{2}t,
\end{align}
completing the proof of \eqref{eq:conv:ub:f:t}, and thus of Theorem \ref{thm:f:t:eod}.
\end{proof}

\begin{proof}[Proof of Lemma \ref{lem:prob:gron}]
For fixed $K>0$, applying the Gr\"onwall lemma to \eqref{jpgnw}, we have
\begin{equation}
\label{eq:con1}
\E\Big(\mcx(t)\mathbbm{1}_{\tau_K>t}\Big)
\leq  \E \mcx(t \wedge \tau_K)
\leq \Big(K \mct + CK^2(\eps + \delta)(1 + \|\sigma\|^{2}t)\Big)
\exp(C K t).
\end{equation}
Then, for any $0<\gamma<1$,
\begin{align}
\label{eq:con2}
\E (\mcx(t))^{\gamma}  &=
\sum_{k=1}^{\infty}\E\Big((\mcx(t))^{\gamma}
\mathbbm{1}_{\tau_{k} > t}  \mathbbm{1}_{\tau_{k-1} \leq t}
\Big)  \leq \sum_{k=1}^{\infty}\Big(\E\Big(\mcx(t) \mathbbm{1}_{\tau_{k} > t}\Big)\Big)^{\gamma}
(\Prb(\tau_{k-1} \leq t))^{1-\gamma} \notag \\
&\leq
(C')^{1- \gamma}
\sum_{k=1}^{\infty}
\Big[k \mct + Ck^2(\eps + \delta)(1 + \|\sigma\|^{2}t)
\Big]^{\gamma}\exp(\gamma C  t k -(1-\gamma)\eta k) \notag \\
&\leq
C_1(\mct^\gamma + (\eps + \delta)^\gamma)\,,
\end{align}
where we have assumed in the last line
$\gamma<\frac{\eta}{\eta+Ct}$ to guarantee convergence of
the series.
\end{proof}

\begin{Rmk}\label{rmk:cor:anal}  We note that the formal limit
system \eqref{eq:AS:th}--\eqref{eq:AS:mag} does not well approximate
\eqref{eq:MHD:vel}--\eqref{eq:MHD:temp} in its velocity and magnetic
components at time $t =0$.  This is due to the singular nature of the limit as $\eps, \delta \to 0$
and is reflected in the phase space mismatch between the two systems.   Such considerations are
behind the derivation and analysis of the so called `corrector systems' for analogous
singular perturbation problems arising in large Prandtl number convection carried out in \cite{Wang2004a, FoldesGlattHoltzRichards2015}.

To obtain a more accurate approximation for \eqref{eq:MHD:vel}--\eqref{eq:MHD:temp}
in our situation we may expand solutions in formal asymptotic series involving multiple time scales
for instance following the method of inner and outer expansions as in e.g. \cite{OMalley1974,KabanovPergamenshchikov2003}.
If $\delta \ll \eps \ll 1$ this leads to a two stage approximation.
By first treating $\delta \approx 0$ in comparison to $\eps$ we obtain
\begin{align*}
\uc &= e^{-t\OQ/\eps}\U(0) + \OQ^{-1} \bfe_3 \thc(t) - e^{-t  \OQ/\eps}\OQ^{-1}\bfe_3\Th(0),	
	\\
\bc &= (-\Delta)^{-1}(\BBB \cdot \nabla) \uc,	
	\\
d\thc &= (- \uc \cdot \nabla \thc + \kappa \Delta \thc)dt + \sigma dW.	
\end{align*}
By next treating $\eps\sim 1$, in comparison to $\delta$ we find
\begin{align}
\eps(\pd{t}\ucc + \uc \cdot \nabla \ucc) &= - A \ucc + \BBB \cdot \nabla \bcc + \bfe_3 \thcc  \label{eq:cor:5}
\\
\bcc(t) &= e^{t\Delta /\del}\B(0) + (-\Delta)^{-1}\BBB \cdot \nabla \ucc(t) - e^{t\Delta /\del}(-\Delta)^{-1} \BBB \cdot \nabla \U(0),
\label{eq:cor:6}
\\
d\thcc  &= (- \ucc \cdot \nabla \thcc + \kappa \Delta \thcc)dt + \sigma dW.
\label{eq:cor:7}
\end{align}
This collection of equations provides a well posed approximation of \eqref{eq:MHD:vel}--\eqref{eq:MHD:temp}.
The accuracy of these approximations can be rigorously established using variations on the analysis carried out in the proof of Theorem~\ref{thm:f:t:eod}
(and see also  \cite{FoldesGlattHoltzRichards2015}).
Note also that analogous approximate systems can be found in the cases $\eps \ll \delta \ll 1$ and when $\eps \sim \delta$.
In any case such approximate equations are not required for our analysis below which focuses on long time asymptotics, and
we omit further details.

\end{Rmk}

\section{Analysis of the Markovian Dynamics of the Formal Limit System}
\label{sec:lim}

In this section we establish conditions guaranteeing the contractive property \eqref{eq:cont}, see point (ii) in the summary of main results.
Recall that \eqref{eq:cont} implies that the limit equation \eqref{eq:AS:th} has at most one invariant measure
which is thus ergodic.  Moreover, this contractive property plays a crucial role in Section~\ref{sec:conv}, where it is used
to establish the convergence of statistically stationary states for \eqref{eq:MHD:vel}--\eqref{eq:MHD:temp} as $\eps,\delta \rightarrow 0$.

As we have already noted in the introduction associating the contraction \eqref{eq:cont} with a version of the H\"{o}rmander bracket condition
does not require any significant new ideas in comparison to \cite{HairerMattingly2008}.  Rather our contribution
here is to show that H\"{o}rmander's condition holds for \eqref{eq:AS:th} with some specific noise configurations.
Indeed, in view of the complex nature of the constitutive law \eqref{eq:smb:u1}--\eqref{eq:smb:u2} this verification requires a non-trivial analysis.

Before proceeding to the main results in the section, we first recall some preliminaries related to the
H\"{o}rmander bracket condition and to Wasserstein metrics.

\subsection*{H\"{o}rmander's Condition}

In our setting and in view of \cite{HairerMattingly06, HairerMattingly2008} (see also \cite{Doob1948, Khasminskii1960}),
the proof of \eqref{eq:cont} reduces to establishing smoothing properties of the Markov semigroup $\{P_t\}_{t\geq 0}$
 associated to \eqref{eq:AS:th}.
Following \cite{HairerMattingly06} and in the sprit of  \cite{Hormander1967} this leads to
a H\"{o}rmander-type bracket condition as follows.

Define, for each $k \in \ZZ^3$ and $x \in \TT^3$,
\begin{align*}
\sigma_k^0(x) := \cos (k\cdot x), \qquad \sigma_k^1(x) := \sin (k\cdot x),
\end{align*}
to be the eigenfunctions of the Laplacian $\TT^3$, and recall that
\begin{align*}
   \sigma dW = \sum_{k \in \ZZ_{+}^3, m\in\{ 0,1\}} \alpha_{k, m} \sigma_k^m dW^{k,m} \,,
\end{align*}
where $\ZZ_{+}^{3}=\{k = (k_1,k_2,k_3) \in \ZZ^{3}:k\neq 0, k_1\geq 0\}$ is the integer lattice half-space
and $\alpha_{k, m} \in \RR$.
Let
\begin{align}
  \mathcal{W}_0 &:= \text{span}\{ \sigma_j^m :  \alpha_{j, m} \not= 0  \}, \notag\\
  \mathcal{W}_n &:= \text{span}\Big\{\{ [[F, \psi], \sigma] = M_u(\psi) \cdot \nabla \sigma + M_u(\sigma) \cdot \nabla \psi :
  					\psi \in \mathcal{W}_{n-1}, \sigma \in \mathcal{W}_0\}
  				\cup \mathcal{W}_{n-1}\Big\}, \label{eq:nth:Hor:space}
\end{align}
where
\begin{equation} \label{f:def}
F(\thl) =  - \kappa \Delta \theta + M_u(\theta) \cdot \nabla \theta,
\end{equation}
and where, for any Fr\' echet differentiable functions (vector fields) $E_1, E_2 : H \to H$,
one defines the Lie bracket $[E_1, E_2] : H \to H$ as
\begin{equation}\label{eq:brak:def}
[E_1, E_2](\theta) = \nabla E_2(\theta) E_1(\theta) - \nabla E_1(\theta) E_2(\theta) \,.
\end{equation}

The H\"ormander-type condition is now stated as follows:
\begin{align}
  \textrm{For every } N > 0 \textrm{ there exists an } n = n(N) \textrm{ such that }  \mathcal{W}_n \supseteq H_N \,,
  \label{eq:Hormander:easy}
\end{align}
where $H_N = \text{span}\{ \sigma_{k}^{m}:k\in \ZZ_{+}^{3},|k|\leq N,m\in\{0,1\}\}$.

\subsection*{Wasserstein Metrics}

In order to clarify how \eqref{eq:Hormander:easy} implies contractivity in the Markovian dynamics
associated to \eqref{eq:AS:th}, we proceed to discuss a few generalities about Wasserstein metrics.
Recall that, given a metric space $(X, \mathfrak{m})$, the associated Wasserstein distance $\mathfrak{W}_{\mathfrak{m}}$
on $\text{Pr}(X)$ is defined by the following equivalent formulations:
\begin{align}\label{eq:equivalent}
  \WW_{\mathfrak{m}}(\mu, \nu) = \inf_{\Gamma \in \mathcal{C}(\mu, \nu)} \int  \mathfrak{m}(u, v) d \Gamma(u,v)
  = \sup_{\| \phi\|_{\text{Lip},  \mathfrak{m}} \leq 1} \left|\int \phi d\mu - \int \phi d \nu \right| \,.
\end{align}
See e.g. \cite{Villani2008}. Here $\mathcal{C}(\mu, \nu)$ denotes the set of couplings between $\mu$ and $\nu$, namely
\begin{align*}
  \mathcal{C}(\mu, \nu) = \{ \Gamma \in \text{Pr}(X \times X): \Gamma(A \times X) = \mu(A),  \Gamma(X \times B) = \nu (B)
  \textrm{ for any } A, B \in \mathcal{B}(X)\},
\end{align*}
and
\begin{align}\label{eq:Lip}
  \| \phi\|_{\text{Lip},  \mathfrak{m}} = \sup_{u \not = v} \frac{|\phi(u) - \phi(v)|}{  \mathfrak{m}(u,v)},
\end{align}
is the Lipschitz semi-norm corresponding to $(X, \mathfrak{m})$.

In our setting we consider the metric $\rho = \rho_{\eta}$ on the phase space $H' = L^2(\TT^3)$ as
\begin{align}\label{eq:m}
   \rho (\theta, \psi) = \inf_{p \in P(\theta, \psi)} \int_0^1 \exp( \eta \| p\|^2) \| p'\| dt
\end{align}
for a suitable $\eta > 0$, where $P(\theta, \psi) = \{ p \in C^1([0,1], H'): p(0) =\theta, p(1) =\psi\}$.  Notice that
\begin{align}
   \| \theta - \psi\|   \leq  \rho (\theta, \psi)  \leq \exp( 2\eta (\|\theta \|^2 + \| \psi  \|^2))  \| \theta - \psi\|.
   \label{eq:semi:eq:rho}
\end{align}

\subsection*{Criteria for Contractivity}
Invoking the abstract machinery developed in \cite{HairerMattingly2008,HairerMattingly2011}, we now
establish contractivity for the Markovian dynamics whenever H\"{o}rmander's condition \eqref{eq:Hormander:easy}
is satisfied.

\begin{Thm}\label{thm:c:w}
 Let $\{P_t\}_{t \geq 0}$ be the Markov semigroup  associated to \eqref{eq:AS:th} and assume that the H\"ormander
 condition  \eqref{eq:Hormander:easy} holds.
 Then $\{P_t\}_{t \geq 0}$ is contractive in $\WW_{\rho}$:
\begin{align}
  \WW_{\rho}(\mu_1 P_t, \mu_2 P_t) \leq C e^{- \gamp t}
  \WW_{\rho}(\mu_1, \mu_2)  \quad \text{ for every } t \geq 0\,,
  \label{eq:distance:contraction:cond}
\end{align}
and hence $\{P_t\}_{t \geq 0}$ possesses a unique ergodic invariant measure.  Here the metric $\rho$ is defined relative
to an $0< \eta \leq \eta_0(\|\sigma\|, \kappa, \nu)$.  This number $\eta_0$ and
the constants $C = C(\|\sigma\|, \kappa, \nu)> 0$, $\gamp = \gamp(\|\sigma\|, \kappa, \nu)>$ are all independent
of $t \geq 0$ and $\mu_1, \mu_2$.
\end{Thm}

\begin{Rmk}
Under the conditions of Theorem~\ref{thm:c:w} the unique invariant measure $\mu$
satisfies certain additional attraction properties: $\mu$ is exponentially mixing, and it obeys a strong law of large numbers and a central limit theorem.
See e.g. \cite{FriedlanderGlattHoltzVicol2014} for more precise analogous statements.
\end{Rmk}

\begin{proof}
The proof essentially follows ideas for the 2D Navier-Stokes equations
in vorticity formulation developed in \cite{HairerMattingly06,HairerMattingly2008,HairerMattingly2011}, so we will be sparing in details.
Indeed, while we are working in 3D, the constitutive law in \eqref{eq:AS:th}--\eqref{eq:AS:mag}
(that is, the functional $M$, cf. \eqref{eq:smb:u1}) has two degrees of smoothing
as compared to the one degree in the Biot-Sawart law.

The bound \eqref{eq:distance:contraction:cond} follows directly from
\cite[Theorem 3.4]{HairerMattingly2008} which requires us to verify three conditions.
The first condition, and the most difficult step in our setting, is a bound for $ \nabla P_t \phi (u)$ closely related to the
asymptotic strong Feller condition and is thus a form of smoothing. This property
follows from \cite[Theorem 8.4]{HairerMattingly2011} by invoking \eqref{eq:Hormander:easy} along with the
Lemmata \ref{lem:mom:bnds:1}--\ref{lem:mom:lin:2} established below.  Note that this is the only
point in the argument where we make use of \eqref{eq:Hormander:easy}.

The second and third conditions needed for \cite[Theorem 3.4]{HairerMattingly2008} are
a form of irreducibility and certain exponential moment bounds. The irreducibility condition
is straightforward in our setting due essentially to the fact that zero is a stable fixed point of the unforced dynamics.
See e.g. \cite{EMattingly2001,ConstantinGlattHoltzVicol2013} for further details. The second moment condition holds due to the
Lemmata \ref{lem:mom:bnds:1}--\ref{lem:mom:lin:2} below.    With
\eqref{eq:distance:contraction:cond} verified, we immediately infer the uniqueness and ergodicity
of invariant measures.  A standard application of the Krylov-Bogoliubov averaging procedure
and compactness yields the existence of invariant measures and thus the proof is now complete.

\end{proof}

\subsection*{Lie Bracket Computations}

We now demonstrate one situation in which \eqref{eq:Hormander:easy} is satisfied, but a variety of other configurations can be
handled with the approach given below.

\begin{Thm}\label{lem:hbc}
The H\" ormander bracket condition \eqref{eq:Hormander:easy}
holds if $\sigma_k^m \in \mathcal{W}_0$, or equivalently $\alpha_{k,m}\neq 0$ for all $k \in \{\bfe_1, \bfe_2, \bfe_3\}$, $m \in \{0, 1\}$.
\end{Thm}

A preliminary step in the proof of Theorem~\ref{lem:hbc} is to establish an iterative process for producing new directions
 that generates the subspaces $\mathcal{W}_n$ specified by \eqref{eq:nth:Hor:space}.
\begin{Lem}\label{lem:one:stp}
\mbox{}
\begin{itemize}
\item[(i)] Suppose that $\sigma_{j}^{m'} \in \mathcal{W}_0$ and
$\sigma_k^m \in \mathcal{W}_n$ for some $j, k \in \ZZ^3_{+}$ and all $m,m' \in \{0, 1\}$.
Then $\sigma_{k + j}^{m}, \sigma_{k - j}^{m} \in \mathcal{W}_{n+1}$ for $m \in \{0, 1\}$ whenever
\begin{align}
|\MMM_u(k) \cdot j| \neq |\MMM_u(j) \cdot k|.
\label{eq:g:n:dir:cond}
\end{align}
\item[(ii)]  Assume that $\sigma_{\bfe_1}^{m}\in \mathcal{W}_0$ and $\sigma_k^{m'} \in \mathcal{W}_{n}$ for some $k = (k_1,k_2, k_3) \in \ZZ^3_{+}$ and all $m, m' \in \{0, 1\}$.  Then  $\sigma_{k \pm \bfe_1}^{m} \in \mathcal{W}_{n+1}$ for $m\in\{0,1\}$, provided
\begin{align}
 | k_{2} (\Omm \cdot k)|k|^2  +  k_1k_3((\BBB \cdot k)^2 + \nu |k|^4)| \neq \frac{|k_3| D(k)}{\nu + (\BBB)_1^2}
 \label{cdt:e1}
\end{align}
\item[(iii)]
Similarly if $\sigma_{\bfe_2}^{m}\in \mathcal{W}_0$ and $\sigma_k^{m'} \in \mathcal{W}_{n}$ for $m, m' \in \{0,1\}$ then $\sigma_{k \pm \bfe_2}^{m'} \in \mathcal{W}_{n+1}$ for $m'\in\{0,1\}$ if
\begin{align}
 |-k_{1} (\Omm \cdot k)|k|^2  +  k_2k_3((\BBB \cdot k)^2 + \nu |k|^4)| \neq \frac{D(k)}{(\nu + (\BBB)_2^2)^2 + \Omm_2^2} |(\nu + (\BBB)_2^2)k_3 - \Omm_2 k_1|
 \label{cdt:e2}
\end{align}
\item[(iv)] Finally whenever $\sigma_{\bfe_3}^{m}\in \mathcal{W}_0$ and $\sigma_k^{m'} \in \mathcal{W}_{n}$ for $m, m' \in \{0,1\}$,
then $\sigma_{k \pm \bfe_3}^{m'} \in \mathcal{W}_{n+1}$ for $m' \in \{0,1\}$ as long as
\begin{equation}
\label{cdt:e3}
k_1^2 + k_2^2 \neq 0.
\end{equation}
\end{itemize}
\end{Lem}

\begin{proof}
Observe that from \eqref{eq:eod} and the definition of the operator $M$ one obtains for each $k \in \ZZ^3_{+}$ and $m \in \{0, 1\}$,
\begin{align}
M_u(\sigma_k^m) = \MMM_u(k) \sigma_k^m \,,
\label{eq:eig:prop}
\end{align}
where $\MMM_u$ is defined in \eqref{eq:smb:u1} and is independent of $m$.
To simplify notation we drop the subscripts $u$ of $M_u$, $\MMM_u$ and instead
write $M$, $\MMM$ for the remainder of the proof.

For the first item, (i), observe that by definition $[[F, \sigma_k^m], \sigma_j^{m'}] \in \mathcal{W}_{n+1}$ for the specified $j, k$ and $m, m' \in \{ 0,1 \}$.
Here recall that the superscript in $\sigma_k^m$ is understood modulo 2,
and $F$ and the Lie bracket are defined by \eqref{f:def}, \eqref{eq:brak:def}.  With
\eqref{eq:eig:prop} we observe that:
\begin{align}
[[F, \sigma_k^m], \sigma_j^{m'}] &=
M(\sigma_k^m) \cdot \nabla \sigma_j^{m'} +
M(\sigma_j^{m'}) \cdot \nabla \sigma_k^m
 \notag\\
&=
 (-1)^{m' + 1} (\MMM(k) \cdot j) \sigma_{k}^m \sigma_j^{m'+1}  +
 (-1)^{m + 1} (\MMM(j) \cdot k) \sigma_{k}^{m+1} \sigma_j^{m'}  \,,
 \label{eq:brak:comp:1}
\end{align}
Using standard trigonometric identities we have
\begin{align*}
\sigma_{k}^n \sigma_j^{n'} = \frac{1}{2}\Big( (-1)^{n n'} \sigma_{k + j}^{n+ n'} +
(-1)^{(n + 1) n'} \sigma_{k - j}^{n + n'}  \Big) \,,
\end{align*}
and consequently \eqref{eq:brak:comp:1} implies
\begin{align}
&[[F, \sigma_k^m], \sigma_j^{m'}]
\notag\\
&= \! \frac{1}{2}
\Big( \!
(-1)^{(m + 1)(m' + 1)} (\MMM(k) \cdot j \! + \! \MMM(j) \cdot k)
 \sigma_{k + j}^{m + m' + 1}  \! + \!
(-1)^{m(m' + 1)}  (\MMM(k) \cdot j \! - \! \MMM(j) \cdot k)
   \sigma_{k - j}^{m + m' + 1}
\Big)\,.
\label{eq:brak:comp:2}
\end{align}
Then with $m, m' \in \{0,1\}$ replaced by $m+1$, $m'+1$,
\begin{align}
&[[F, \sigma_k^{m+1}], \sigma_j^{m'+1}]
\notag\\
&= \frac{1}{2}
\Big(
(-1)^{m m'} (\MMM(k) \cdot j + \MMM(j) \cdot k)
 \sigma_{k + j}^{m + m' + 1}
+
(-1)^{(m+1) m'}  (\MMM(k) \cdot j -  \MMM(j) \cdot k)
   \sigma_{k - j}^{m + m' + 1}
\Big)\,.
\label{eq:brak:comp:3}
\end{align}
Combining \eqref{eq:brak:comp:2}, \eqref{eq:brak:comp:3} we find that
$\sigma_{k + j}^{m + m' + 1}, \sigma_{k - j}^{m + m' + 1} \in \mathcal{W}_{n+1}$, provided that the matrix
\begin{equation}
\left(
\begin{array}{cc}
(-1)^{(m + 1)(m' + 1)} (\MMM(k) \cdot j + \MMM(j) \cdot k)
&  (-1)^{m m'} (\MMM(k) \cdot j + \MMM(j) \cdot k)\\
(-1)^{m(m' + 1)}  (\MMM(k) \cdot j -  \MMM(j) \cdot k)
& (-1)^{(m+1) m'}  (\MMM(k) \cdot j -  \MMM(j) \cdot k)
\end{array}
\right)
\label{eq:mat1}
\end{equation}
is invertible, which is true exactly when \eqref{eq:g:n:dir:cond} holds, and (i) follows.

The remaining items (ii)--(iv) now follow from (i) and \eqref{eq:smb:u1}--\eqref{eq:smb:u2} after some direct computations.
Here in particular note \eqref{cdt:e1} uses the fact that
$\Omm_1 = 0$. The proof of Lemma~\ref{lem:one:stp} is complete.
\end{proof}

\begin{proof}[Proof of Theorem~\ref{lem:hbc}.]

We begin by making the preliminary observation that
\begin{align}
\text{$\forall N > 0, \exists K = K(N, \nu, \Omm)$ such that \eqref{cdt:e1}, \eqref{cdt:e2}
both hold whenever $|k_1|, |k_2| \leq N, k_3 \geq K$.}
\label{eq:far:out}
\end{align}
Indeed, for any fixed $|k_1|, |k_2| \leq N$,
the left hand side of \eqref{cdt:e1}, \eqref{cdt:e2} grows as
$k_3^5$ while the right hand side grows as $k_3^9$ as $k_3 \to \infty$.

We proceed to establish the desired result
by showing that
\begin{align}
\text{for any $N > 0$ there is an $n$ such that
$\sigma_k^m \in \mathcal{W}_n$ for any $|k_1|, |k_2|, |k_3| \leq N$
and any $m \in \{0, 1\}$. }
\label{eq:Hor:equiv}
\end{align}
Fix $N > 0$.  Starting from $k = (1, 0, 0)$ and repeatedly applying Lemma~\ref{lem:one:stp}, (iv), we obtain an $n_1$ such
that $\sigma_k^m \in \mathcal{W}_{n_1}$ for $k = (1, 0, K)$ and each of $m \in\{0,1\}$. Here,
$K > 0$ is as in \eqref{eq:far:out} corresponding to the given $N >0$.

Next, by Lemma~\ref{lem:one:stp}, (ii) and (iii), we obtain an $n_2 > n_1$ such that
$\sigma_k^m \in \mathcal{W}_{n_2}$ for every $k = (k_1, k_2, K)$ with $|k_1|, |k_2| \leq N$ and $k_3 = K$.
From here, applying Lemma~\ref{lem:one:stp}, (iv), we obtain $n_3 > n_2$ so that
$\sigma_k^m \in \mathcal{W}_{n_3}$,
whenever $|k_1|, |k_2|, |k_3| \leq N$, $m \in \{0,1\}$ and $k \not \in \text{span}\{ \bfe_3\}$.

Finally, to obtain the directions corresponding to $\text{span}\{ \bfe_3\}$, consider, for any $l \in \ZZ \setminus \{0\}$,
elements $\sigma_k^m \in \mathcal{W}_{n_3}$ of the form $k = (\pm 1, 0, l)$ with $m = 0,1$.
For any such $k$, noting that $ k_2 = 0$, the condition \eqref{cdt:e1} reduces to
\begin{align}
(\BBB \cdot k)^2 + \nu |k|^4
\neq \frac{D(k)}{\nu + (\BBB)_1^2}
\,.
\label{eq:move:to:ax:cond}
\end{align}
Now observe that the sign in $k_1$ may be chosen such that $(\BBB)_1^2 \leq (\BBB \cdot k)^2$. Indeed, one may take $k_1$ so that
$\textrm{sign} (k_1 (\BBB)_1) = \textrm{sign} (k_3 (\BBB)_3)$ or, in the case that either $(\BBB)_1 = 0$
or $(\BBB)_3 = 0$, one may choose any sign for $k_1$.  With this choice, \eqref{eq:move:to:ax:cond} holds
since
\begin{align*}
(\BBB \cdot k)^2 + \nu |k|^4
<   \frac{1}{\nu + (\BBB)_1^2}((\BBB \cdot k)^2 + \nu |k|^4)^2
\leq \frac{1}{\nu + (\BBB)_1^2} D(k)\,.
\end{align*}
Thus, we obtain an $n \geq n_3$ such that \eqref{eq:Hor:equiv} holds for the given $N$.
The proof of Theorem~\ref{lem:hbc} is complete.
\end{proof}

\section{Convergence of Statistically Steady States}
\label{sec:conv}

In this final section we combine the finite time bounds obtained from Section~\ref{sec:fta} with the contraction estimate
\eqref{eq:distance:contraction:cond} of Section~\ref{sec:lim}, to establish the convergence of statistically invariant
states for \eqref{eq:MHD:vel}--\eqref{eq:MHD:temp} to the unique invariant measure of \eqref{eq:AS:th} (see
(iv) in the summary of main results).  As we discussed in the introduction, one of the novelties here is that
our convergence analysis applies to the extend phase space $H$, not just to the temperature component,
as was the case in our previous work \cite{FoldesGlattHoltzRichards2015}.

In order to precisely state the main result of this section we introduce some notation.
Recall that the metric $\rho$ on $L^2(\mathbb{T}^3)$ is defined in \eqref{eq:m}.  Here the parameter $\eta > 0$
appearing in \eqref{eq:m} is taken to be a suitably small value so that \eqref{eq:distance:contraction:cond} applies.
Using $\rho$ we define a new distance on $H$ according to
\begin{align*}
\tilde{\rho}((\U,\B,\Th),(\tilde{\U},\tilde{\B},\tilde{\Th})) =  \|\U-\tilde{\U}\|_{H^1} + \|\B-\tilde{\B}\|_{H^1} + \rho(\Th,\tilde{\Th}).
\end{align*}
Following the notation in \eqref{eq:equivalent}, $\tilde{\rho}$ induces a Wasserstein metric on $H$ which we denote by
$\WW_{\tilde{\rho}}$.

In what follows it will be useful to consider the set of `observables'
\begin{align}
    V(H) = V_{\eta}(H) := \left\{\phi \in C^1(H):
     [\phi]_{\eta} < \infty \right\}\,,
     \label{eq:v:eta}
\end{align}
where $[ \cdot ]_{\eta}$ denotes
\begin{align*}
  [\phi]_{\eta} :=  &\sup_{(u,b,\theta) \in (H^1)^{2}\times H'} \left[
   \sup_{\zeta_1\in H^1,\|\zeta_1 \|_{H^1} =1} | \nabla _{u}\phi( u,b,\theta) \cdot\zeta_1 |
   +   \sup_{\zeta_2\in H^1,\|\zeta_2 \|_{H^1} =1} | \nabla _{b}\phi( u,b,\theta) \cdot\zeta_2 | \right.  \\
  & \left.\hspace{2in} + \exp(-\eta \|\theta \|)\sup_{\xi\in H',\|\xi \| =1} | \nabla_{\theta} \phi( u,b,\theta)\cdot \xi | 
\right].
\end{align*}
It can be verified with a straight-forward proof (see \cite[Proposition 4.1]{HairerMattingly2008}) that, for any $\phi \in C^1(H)$,
\begin{align}
\label{eq:vbound}
\| \phi \|_{Lip, \tilde{\rho}} \leq C [ \phi ]_{ \eta}.
\end{align}
This bound is useful for translating the convergence of measures relative to $\WW_{\tilde{\rho}}$ to the convergence of observables from
$V(H)$. See \eqref{eq:conv:obs} below.

Let $L:H'\rightarrow H$ be the extension operator associated to \eqref{eq:AS:vel}--\eqref{eq:AS:mag} given by
\begin{align*}
L(\theta) = (M(\theta),\theta).
\end{align*}
Also, the projection operator $\R: H \to H'$ will associate elements in $H$ to their $\theta$ component,
\begin{align*}
  \R(u,b, \theta) = \theta.
\end{align*}
Recall that for any measure $\mu$ and function $F$, the push-forward of $\mu$ under $F$ is given by $F_{\#} \mu  = \mu \circ F^{-1}$.
To simplify notation, for a measure $\mu_1$ on $H'$, we will use $L\mu_1=L_{\#}\mu_1$ to denote the extended measure on $H$, and similarly for $\mu_2$ on $H$ we will write $\Pi\mu_2=\Pi_{\#}\mu_2$ for the projected measure on $H'$ (i.e. its marginal in the $\theta$ component).

\begin{Thm}\label{thm:SSS:conv}
  Suppose that the conditions imposed in Theorem~\ref{thm:c:w} are satisfied.   For any $ \eps,\delta > 0$, let $\mu_{\eps,\delta}$ be a statistically
  invariant state of \eqref{eq:MHD:vel}--\eqref{eq:MHD:temp} satisfying the uniform moment condition \eqref{eq:un:wk:stat:bnd}
  and let $\mu_0$ be the unique invariant measure of \eqref{eq:AS:th}.
  Then, there exists $\tilde{\gamma} = \tilde{\gamma}(\nu, \kappa, \| \sigma \|)$, $C = C(C_0,\nu, \kappa, \| \sigma \|)$ which are independent of $\eps, \delta > 0$, such that
  \begin{align}
    \WW_{\tilde{\rho}}( \mu_{\eps,\delta}, L \mu_0)  \leq C (\eps + \delta)^{\tilde{\gamma}}
    \label{eq:ext:conv:sss}
  \end{align}
  for every $\eps, \delta > 0$.  This implies that for statistically invariant states $(\U,\B,\Th)$ and $\theta$ distributed respectively
  as $\mu_{\eps,\delta}$ and $\mu_0$,
  \begin{align}
  |\E(\phi(\U,\B,\Th) - \phi(L\theta))| \leq C[ \phi ]_{\eta}(\eps + \delta)^{\tilde{\gamma}}
  \label{eq:conv:obs}
  \end{align}
  for any $\phi\in V(H)$.
\end{Thm}

\begin{Rmk}
The estimate \eqref{eq:conv:obs} indicates that, at statistical equilibrium, observations of $\Th$ are well approximated by observations of $\theta$,
which in turn provide good estimates on observations of $(\U,\B)$ through $M(\theta)$.  That is, using \eqref{eq:conv:obs}
we can approximate $(\U,\B)$ using observations of $\Th$ from statistically stationary states.
\end{Rmk}

The proof of Theorem~\ref{thm:SSS:conv} makes use of a second metric on $L^2(\TT^3)$ defined as
\begin{align*}
\rho^{*}(\theta,\tilde{\theta}) = \tilde{\rho}(L(\theta),L(\tilde{\theta})),
\end{align*}
and its associated Wasserstein distance $\WW_{\rho^*}$.

\begin{Lem} \label{lem:rhost:equiv}
The metric $\rho^*$ is equivalent to $\rho$ and hence $\WW_{\rho^*}$ and $\WW_{\rho}$ are also equivalent.    Moreover,
\begin{align}
   \WW_{\tilde{\rho}}(L \mu_1,L\mu_2) \leq \WW_{\rho^{*}}(\mu_1 , \mu_2) \leq C\WW_{\tilde{\rho}}(L \mu_1,L\mu_2),
   \label{eq:rhost:equiv:wd}
\end{align}
for any $\mu_1,\mu_2 \in Pr(L^2(\TT^3))$.
\end{Lem}

\begin{proof}[Proof of Lemma \ref{lem:rhost:equiv}]
Invoking the smoothing properties of the constitutive law $M=(M_u,M_b)$ along with \eqref{eq:semi:eq:rho}, we find
\begin{align}
\rho(\theta,\tilde{\theta})
	\leq \rho^*(\theta,\tilde{\theta})
	= \|M_u(\theta -\tilde{\theta})\|_{H^1} + \|M_b(\theta-\tilde{\theta})\|_{H^1} + \rho(\theta,\tilde{\theta})
	\leq C\|\theta - \tilde{\theta}\| + \rho(\theta,\tilde{\theta}) \leq C\rho(\theta,\tilde{\theta}).
\label{eq:equiv1}
\end{align}
It follows directly from the definition
that the corresponding Wasserstein distances $\WW_{\rho}$ and $\WW_{\rho^*}$ are equivalent as well; cf. \eqref{eq:equivalent}.

To establish \eqref{eq:rhost:equiv:wd} for each $k > 0$, we consider
\begin{align*}
  T_1^k &:= \{  \phi: L^2(\mathbb{T}^3) \to \RR: \| \phi \|_{Lip, \rho^*} \leq k \},\\
  T_2^k &:= \{ \phi: L^2(\mathbb{T}^3) \to \RR: \phi(\theta) = \psi(L \theta) \textrm{ for some } \psi: H \to \RR \textrm{ with } \| \psi \|_{Lip, \tilde{\rho}} \leq k\}.
\end{align*}
We claim that for each $k > 0$ one has $T_2^k \subset T_1^k \subset T_2^{Ck}$, where $C$ is as in \eqref{eq:equiv1}. Indeed, if $\phi \in T_2^k$,
then
\begin{equation*}
\frac{|\phi(\theta) - \phi(\tilde{\theta})|}{\rho^{*}(\theta, \tilde{\theta})} =
\frac{|\psi(L(\theta)) - \psi(L(\tilde{\theta}))|}{\tilde{\rho}(L(\theta), L(\tilde{\theta}))}
\leq k, \qquad \textrm{for each } \theta \neq \tilde{\theta}\,,
\end{equation*}
and therefore $\phi \in T_1^k$. On the other hand if
$\phi \in T_1^k$ we define $\psi := \phi \circ \Pi$ (clearly $\phi = \psi \circ L$) and calculate
\begin{equation*}
\frac{|\psi(u, b, \theta) - \psi(\tilde{u}, \tilde{b}, \tilde{\theta})|}{\tilde{\rho}((u, b, \theta), (\tilde{u}, \tilde{b}, \tilde{\theta}))}
= \frac{|\phi(\theta) - \phi(\tilde{\theta})|}{\tilde{\rho}((u, b, \theta), (\tilde{u}, \tilde{b}, \tilde{\theta}))} \leq
\frac{|\phi(\theta) - \phi(\tilde{\theta})|}{\rho(\theta, \tilde{\theta})} \leq C \frac{|\phi(\theta) - \phi(\tilde{\theta})|}{\rho^*(\theta, \tilde{\theta})}
= Ck,
\end{equation*}
and we showed $\phi \in T_2^{Ck}$.
Finally, for any $\mu_1, \mu_2 \in Pr(L^2(\TT^3))$,
\begin{multline*}
   \WW_{\tilde{\rho}}(L \mu_1,L\mu_2)
   = \sup_{\|\psi \|_{Lip, \tilde{\rho}} \leq 1} \left| \int \psi(L\theta) d \mu_1(\theta)  - \int   \psi(L\theta) d \mu_2(\theta) \right|\\
   = \sup_{\phi \in T_2^1} \left| \int \phi(\theta) d \mu_1(\theta)  - \int   \phi(\theta) d \mu_2(\theta) \right|
   \leq \sup_{\phi \in T_1^1} \left| \int \phi(\theta) d \mu_1(\theta)  - \int   \phi(\theta) d \mu_2(\theta) \right|
   = \WW_{\rho^*}( \mu_1, \mu_2),
\end{multline*}
and
\begin{align*}
\WW_{\rho^*}( \mu_1, \mu_2) &= \sup_{\phi\in T_1^1}\left| \int \phi(\theta) d \mu_1(\theta)  - \int   \phi(\theta) d \mu_2(\theta) \right|
\leq \sup_{\tilde{\phi}\in T_2^C}\left| \int \tilde{\phi}(\theta) d \mu_1(\theta)  - \int   \tilde{\phi}(\theta) d \mu_2(\theta) \right| \\
&= C \sup_{\tilde{\phi}\in T_2^1}\left| \int \tilde{\phi}(\theta) d \mu_1(\theta)  - \int   \tilde{\phi}(\theta) d \mu_2(\theta) \right|
=
C \WW_{\tilde{\rho}}(L\mu_1,L\mu_2),
\end{align*}
as desired.
\end{proof}

By combining Lemma \ref{lem:rhost:equiv} with Theorem \ref{thm:c:w} we obtain the following `lifted' contraction property.

\begin{Cor}
\label{cor:ext:conv}
Let $\{P_t\}_{t \geq 0}$ be the Markov semigroup  associated to \eqref{eq:AS:th} and assume that the H\"ormander
 condition  \eqref{eq:Hormander:easy} holds.
 Then $\{P_t\}_{t \geq 0}$ is satisfies the following contractive property in $\WW_{\tilde{\rho}}$:
\begin{align}
\label{eq:conL}
\WW_{\tilde{\rho}}(L(\mu_1P_t),L(\mu_2P_t)) \leq Ce^{-\gamp t}\WW_{\tilde{\rho}}(L\mu_1,L\mu_2) \quad \text{for every }\ t\geq 0,
\end{align}
where the constants $C = C(\|\sigma\|, \kappa, \nu)> 0$, $\eta = \eta(\|\sigma\|, \kappa, \nu) >0$, $\gamp = \gamp(\|\sigma\|, \kappa, \nu)>0$ are all independent of $t \geq 0$ and $\mu_1, \mu_2$.
\end{Cor}

\begin{proof}[Proof of Corollary~\ref{cor:ext:conv}]
By combining Lemma~\ref{lem:rhost:equiv} with \eqref{eq:distance:contraction:cond}, we find that
\begin{align*}
\WW_{\tilde{\rho}}(L(\mu_1P_t),L(\mu_2P_t)) \leq C\WW_{\rho^{*}}(\mu_1P_t , \mu_2P_t) \leq C \WW_{\rho}(\mu_1P_t,\mu_2P_t ) \\
\leq Ce^{-\gamp t}\WW_{\rho}(\mu_1,\mu_2) \leq  Ce^{-\gamp t}\WW_{\rho^{*}}(\mu_1,\mu_2) \leq Ce^{-\gamp t}\WW_{\tilde{\rho}}(L\mu_1,L\mu_2),
\end{align*}
for any  $\mu_1,\mu_2 \in Pr(L^2(\TT^3))$, where $C$ is independent of $t$ and $\mu_1, \mu_2$.
\end{proof}

We now proceed with the proof of Theorem~\ref{thm:SSS:conv}.

\begin{proof}[Proof of Theorem~\ref{thm:SSS:conv}]

We invoke the invariance of $\mu_0$ under $P_t$ and apply \eqref{eq:conL} to infer for any $t, t_0 \geq 0$,
\begin{align}
\WW_{\tilde{\rho}}(\mu_{\eps,\delta},L\mu_0)&=\WW_{\tilde{\rho}}(\mu_{\eps,\delta},L( \mu_0P_{t+t_0}))  \notag \\
	&\leq \WW_{\tilde{\rho}}(\mu_{\eps,\delta},L((\Pi \mu_{\eps,\delta})P_{t + t_0}))
		+ \WW_{\tilde{\rho}}(L((\Pi \mu_{\eps,\delta})P_{t + t_0 }),L (\mu_0P_{t + t_0})) \notag \\
	&\leq \WW_{\tilde{\rho}}(\mu_{\eps,\delta},L((\Pi \mu_{\eps,\delta})P_{t + t_0})) +
		Ce^{-\gamp t_0}\WW_{\tilde{\rho}}(L((\Pi \mu_{\eps,\delta})P_{t}),L(\mu_0 P_{t})) \notag \\
	&\leq \WW_{\tilde{\rho}}(\mu_{\eps,\delta},L((\Pi \mu_{\eps,\delta})P_{t + t_0})) +
		Ce^{-\gamp t_0}\left[\WW_{\tilde{\rho}}(\mu_{\eps,\delta},L((\Pi \mu_{\eps,\delta})P_t))+\WW_{\tilde{\rho}}(\mu_{\eps,\delta},L\mu_0)\right]. %\notag\\	
		\notag
\end{align}
By taking $t_0=t_0(\|\sigma\|,\kappa,\nu)$ large enough such that $Ce^{-\gamp t_0}=\frac{1}{2}$, we find that for $t \geq 0$,
\begin{align}		
	\WW_{\tilde{\rho}}(\mu_{\eps,\delta},L\mu_0)
		\leq 2 \WW_{\tilde{\rho}}(\mu_{\eps,\delta},L((\Pi \mu_{\eps,\delta})P_{t + t_0}))
			+ C\WW_{\tilde{\rho}}(\mu_{\eps,\delta},L((\Pi \mu_{\eps,\delta})P_{t})).
	\label{eq:tri:bnd:1}
\end{align}
Using the coupling definition of the Wasserstein metric in \eqref{eq:equivalent} it follows that for any $s \geq 0$,
\begin{align}
\WW_{\tilde{\rho}}(\mu_{\eps,\delta},L((\Pi \mu_{\eps,\delta})P_{s}))
&\leq \E\tilde{\rho}((\U,\B,\Th)(s),L(\theta(s))) \notag\\
 &=  \underbrace{\E(\|\U(s)-M_{u}(\theta(s))\|_{H^1}
+ \|\B(s)-M_{b}(\theta(s))\|_{H^1})}_{\textstyle I_1(s)} + \underbrace{\E \rho(\Th(s),\theta(s))}_{\textstyle I_2(s)},
	\label{eq:tri:bnd:2}
\end{align}
where $(\U,\B,\Th) \sim \mu_{\eps, \delta}$ denotes the stationary solution to \eqref{eq:MHD:vel}--\eqref{eq:MHD:temp} provided by
Proposition \ref{prop:weak:sol:full}, and $\theta$ denotes the solution to \eqref{eq:AS:th}
with initial condition distributed according to the stationary state $\Th = \Th(\eps,\del) \sim \Pi \mu_{\eps, \delta}$.
%On the other hand
%\begin{align}
%\WW_{\tilde{\rho}}(L_{\#}(\Pi_{\theta} \mu_{\eps,\delta}),\mu_{\eps,\delta})
%	&\leq
%   \WW_{\tilde{\rho}}(L_{\#}(\Pi_{\theta} \mu_{\eps,\delta}),L_{\#}((\Pi_{\theta} \mu_{\eps,\delta}) P_{t}))
%				+  \WW_{\tilde{\rho}}(L_{\#}((\Pi_{\theta} \mu_{\eps,\delta}) P_{t}),\mu_{\eps,\delta}) \notag\\
%	&\leq
%   C\WW_{\rho^*}(\Pi_{\theta} \mu_{\eps,\delta},(\Pi_{\theta} \mu_{\eps,\delta}) P_{t})
%				+  \WW_{\tilde{\rho}}(L_{\#}((\Pi_{\theta} \mu_{\eps,\delta}) P_{t}),\mu_{\eps,\delta}) 	\notag\\
%    &\leq C\WW_{\rho}(\Pi_{\theta} \mu_{\eps,\delta},(\Pi_{\theta} \mu_{\eps,\delta}) P_{t})
%				+  \WW_{\tilde{\rho}}(L_{\#}((\Pi_{\theta} \mu_{\eps,\delta}) P_{t}),\mu_{\eps,\delta})	\notag\\	
%    &\leq  C\E \rho(\Th(t),\theta(t))
%				+  \WW_{\tilde{\rho}}(L_{\#}((\Pi_{\theta} \mu_{\eps,\delta}) P_{t}),\mu_{\eps,\delta}),												\label{eq:tri:bnd:3}
%\end{align}
%and we can control the second term on the right-hand side of \eqref{eq:tri:bnd:3} using \eqref{eq:tri:bnd:2}.
By combining the estimates \eqref{eq:tri:bnd:1}--\eqref{eq:tri:bnd:2} and integrating in $t$ over the interval $[0, t_0]$ we infer that
\begin{align} \label{eq:tri:bnd:f:21}
	\WW_{\tilde{\rho}}(\mu_{\eps,\delta},L\mu_0)
	&\leq  \frac{C}{t_0} \int_{0}^{t_0} (I_1(t + t_0) + I_2(t + t_0) + I_1(t) + I_2(t)) \, dt =
	C  \int_{0}^{2t_0} (I_1(t) + I_2(t)) \, dt \,.
\end{align}
%Integrating in $t$ gives
%\begin{align}
%	\WW_{\tilde{\rho}}(\mu_{\eps,\delta},L_{\#}\mu_0)
%	&\leq  \frac{C}{t - t_0}\E \int_{t_0}^t (\|\U(s)-M_{u}(\theta(s))\|_{H^1}  + \|\B(s)-M_{b}(\theta(s))\|_{H^1}  + \rho(\Th(s),\theta(s)))ds  \notag\\
%            &:= I_1 + I_2 + I_3.
%            \label{eq:tri:bnd:f:1}
%\end{align}
where we have absorbed $t_0$ into the constant $C=C(\|\sigma\|,\kappa,\nu)$ in the last line.
Applying the Cauchy-Schwarz inequality, \eqref{eq:conv:ub:f:t} with $\theta(0)$ distributed as $\Th(0)$, and \eqref{eq:un:wk:stat:bnd},
we obtain a bound for $I_1$ as
\begin{align}
  \int_0^{2t_0} I_1(t) \, dt
   &\leq C \left(\E \int_0^{2t_0}  \left(\|\U(t)-M_{u}(\theta(t))\|_{H^1}^2 + \|\B(t)-M_{b}(\theta(t))\|_{H^1}^2 \right) dt\right)^{1/2} \notag\\
   &\leq C \exp(\eta t_0 \| \sigma\|^2_{L^3}) \left((\epsilon + \delta)^{\gamma} +  \left(\E (\eps\|\U(0) -M_{u}(\theta(0))\|^2 + \delta\|\B(0)-M_{b}(\theta(0))\|^2
   ) \right)^{\gamma} \right)^{1/4}
   \notag\\
   &\leq C \exp(\eta t_0 \| \sigma\|^2_{L^3}) (\epsilon + \delta)^{\gamma/4} \left(\E (\|\U(0) \|^2 + \|\B(0)\|^2  + \|\Th(0)\|^2 + 1) \right)^{\gamma/4} \notag\\
   &\leq C \exp(\eta t_0 \| \sigma\|^2_{L^3}) (\epsilon + \delta)^{\gamma/4} \,,
   \label{eq:tri:bnd:f:12}
\end{align}
where $C$ in the last line depends on $C_0$ given in \eqref{eq:ft:d:a:1}, but is independent of $t$, $\eps$, and $\delta$.
%Note that to obtain the last line of \eqref{eq:tri:bnd:f:2}, we used bounds on the constitutive law for \eqref{eq:AS:vel}--\eqref{eq:AS:mag},
%and the assumption that $\theta(0)$ is distributed as $\Th(0)$.
Similarly, using a property of the metric $\rho$, \eqref{eq:semi:eq:rho}, and the bounds \eqref{eq:conv:th:f:t},
\eqref{eq:un:wk:stat:bnd} we have
\begin{align}
   \int_0^{2t_0} I_2(t) \, dt &\leq  C  \E \int_{0}^{2t_0}  \exp( 2 \eta (\| \Th(t)\|^2 + \|\theta(t)\|^2))  \| \Th(t) -\theta(t) \| dt \notag \\
   &\leq C \left( \E \exp\Big(4 \eta
   \sup_{t\in[0,2t_0]} (\| \Th(t)\|^2 + \| \theta(t)\|^2)\Big) \right)^{1/2}
	     \left( \E \sup_{t \in [0,2t_0]} \| \Th(t) -\theta(t) \|^2 \right)^{1/2}\notag \\
    &\leq  C  \exp( 10\eta t_0 \| \sigma\|^2_{L^3})
                (\epsilon + \delta)^{\gamma/4} \left(\E (\|\U(0) \|^2 + \|\B(0)\|^2  + \|\Th(0)\|^2) \right)^{\gamma/4} \notag \\
    &\leq   C  \exp( 10\eta t_0 \| \sigma\|^2_{L^3})
                (\epsilon + \delta)^{\gamma/4},
       \label{eq:tri:bnd:f:3}
\end{align}
where the constant $C$ depends on $C_0$, but is independent of $t_0, \eps$, and $\delta$.
%Using that $(\U, \B, \Th)$ is statistically stationary and referring to \eqref{avebuth}, we find that for any $T > 0$,
%\begin{align}
%  \E (\|\U(0) \|^2 + \|\B(0)\|^2 + \|\Th(0)\|^2)
%  	&= \frac{1}{T} \E  \int_0^T( \|\U(s) \|^2 + \|\B(s)\|^2 + \|\Th(s)\|^2) ds \notag\\
%	&\leq  \frac{C}{T}\E (\eps \|\U(0)\|^2 + \delta  \|\B(0)\|^2  + \|\Th(0)\|^2) + C\| \sigma\|^2, \notag
%\end{align}
%where $C$ is independent of $T, \eps$ and $\delta$.  We thus infer that
%\begin{align}
%   \E (\|\U(0) \|^2 + \|\B(0)\|^2 + \|\Th(0)\|^2)  \leq C\| \sigma\|^2,
%   \label{eq:tri:bnd:f:4}
%\end{align}
%so that this quantity is bounded independently of $\eps, \delta$.
Combining \eqref{eq:tri:bnd:f:21} with \eqref{eq:tri:bnd:f:12}--\eqref{eq:tri:bnd:f:3}
yields \eqref{eq:ext:conv:sss}.  The convergence of observables as in \eqref{eq:conv:obs} follows from \eqref{eq:ext:conv:sss} by using the definition of 
$\WW_{\tilde{\rho}}$ (see \eqref{eq:equivalent}) and applying \eqref{eq:vbound}.
\end{proof}

\appendix

\section{Uniform Moment Bounds}

This section collects various moment bounds related to the full system \eqref{eq:MHD:vel}--\eqref{eq:MHD:temp}
as well as the limit active scalar equation \eqref{eq:AS:th} which we use repeatedly in the analysis
above.

We begin with some $L^p$-bounds on the limit equation.
\begin{Lem}\label{lem:mom:bnds:1}
Let $\thl$ be a solution of \eqref{eq:AS:th} with
$\mathcal{F}_0$ adapted initial condition $\thl_0$.
Then, for any $p \geq 2$ there exists constants $\eta_0 = \eta_0(p, \kappa, \|\sigma\|_{L^p}) > 0$, $C = C(p,\kappa)$
and $\alpha=\alpha(p,\kappa)$,
such that for each $\eta \leq \eta_0$
\begin{align}
    \E \exp\left( \eta \sup_{s \in [0,t]} \|\thl\|^2_{L^p} + \eta \int_0^t \|\thl\|^2_{L^p}ds \right) \leq
    C\E \exp( \eta (\|\thl_0\|^2_{L^p} + t \|\sigma\|^2_{L^p})).
       \label{eq:mod:ad:1}
\end{align}
Moreover,
\begin{align}
    \E \exp\left( \eta \left( \sup_{s \in [0,t]} \|\thl\|^2 + \kappa \int_0^t \| \nabla \thl \|^2 ds \right)\right) \leq  C\E \exp( \eta (\|\thl_0\|^2 + t \|\sigma\|^2)).
       \label{eq:mod:ad:2}
\end{align}
Finally, for any $t \geq 0$
\begin{align}
    \E \exp\left( \eta \|\thl(t)\|^2_{L^p} \right) \leq  C\E \exp( \eta (e^{-\alpha t} \|\thl_0\|_{L^p}^2 +  C\|\sigma\|_{L^p}^2)).
       \label{eq:mod:ad:3}
\end{align}
\end{Lem}

The $L^p$-estimates \eqref{eq:mod:ad:1} and \eqref{eq:mod:ad:3} are established in \cite[Proposition 3.1]{FoldesGlattHoltzRichardsWhitehead2015} for a more
general divergence-free drift diffusion equation
supplemented with mixed Dirichlet-periodic boundary conditions.
Indeed, recall that $\xi=\thl$ satisfies
\begin{align}
   d \xi + (v \cdot \nabla \xi - \kappa \Delta \xi)dt = \sigma dW, \quad \xi(0) =  \xi_0\,,
   \label{eq:mod:ad:0}
\end{align}
with $\|v\|_{H^2} \leq C \|\xi\|$ (see \eqref{eq:smb:u1}).  Then \eqref{eq:mod:ad:1} and
\eqref{eq:mod:ad:3} follow precisely as in the proof of \cite[Proposition 3.1]{FoldesGlattHoltzRichardsWhitehead2015}
 with $\tilde{Ra} = 0$, making obvious adjustments to work on the domain $\mathcal{D}=\TT^3$.

Similarly, an estimate of the form \eqref{eq:mod:ad:2} follows immediately from \cite[Proposition 3.3]{FoldesGlattHoltzRichardsWhitehead2015} for the active scalar equation obtained in the infinite Prandtl limit of the Boussinesq system.  This equation bears structural
similarity to \eqref{eq:AS:th}, and the proof of \eqref{eq:mod:ad:2} for \eqref{eq:AS:th} can be obtained with essentially the same argument, under a simplifying assumption that $\tilde{Ra} = 0$, and with obvious changes for the domain $\mathcal{D}=\TT^3$.

Next, we need an estimate for the first variation
$J_{s, t}\xi$  of
\eqref{eq:AS:th}.  That is, for the solution of
\begin{equation}\label{eq:lln}
   \partial_t \zeta + M_u(\zeta) \cdot \nabla \thl + M_u(\thl) \cdot\nabla \zeta - \kappa \Delta \zeta = 0, \quad \zeta(s) =  \xi\,,
\end{equation}
where $\xi \in L^2(\TT^3)$, $s < t$,  and $M_u$ being as above (see \eqref{eq:smb:u1}).

\begin{Lem}\label{lem:mom:lin}
For each $p > 0$ and $\eta > 0$ there exists $C$ such that
\begin{align*}
\E \sup_{s, t \in [0, 1]} \|J_{s, t}\xi\|^p \leq C \exp(\eta \|\thl_0\|^2) \|\xi\|^p \,.
\end{align*}
\end{Lem}

We also need a similar estimate for the second variation $J^{2}_{s, t}(\xi, \xi')$ of
\eqref{eq:AS:th}, that is, for the solution of
the  problem
\begin{equation}\label{eq:lln2}
   \partial_t \zeta + M_u(\zeta) \cdot \nabla \thl + M_u(\thl) \nabla \zeta + M_u(J_{0, t} \xi) \cdot \nabla J_{0, t} \xi' +
   M_u(J_{0, t} \xi') \cdot \nabla J_{0, t} \xi  - \kappa \Delta \zeta = 0, \quad \zeta(s) =  0 \,,
\end{equation}
where $\xi, \xi' \in L^2(\TT^3)$ and $s < t$.

\begin{Lem}\label{lem:mom:lin:2}
For each $p > 0$ and $\eta > 0$ there exists $C$ such that
\begin{align*}
\E \sup_{s, t \in [0, 1]} \|J^{(2)}_{s, t}(\xi, \xi')\|^p \leq C \exp(\eta \|\thl_0\|^2) \|\xi\|^p \|\xi'\|^p \,.
\end{align*}
\end{Lem}

The proofs of Lemma \ref{lem:mom:lin} and Lemma \ref{lem:mom:lin:2}
follow ideas developed for the 2D Navier-Stokes equations in vorticity formulation
(see e.g. \cite{HairerMattingly06}).  The difference here is that we work in 3D, but have superior smoothing properties of the
constitutive law \eqref{eq:smb:u1}.  The required modifications are straightforward, and we omit details.

Next, we prove for the system \eqref{eq:MHD:vel}--\eqref{eq:MHD:temp} analogous results to those in Lemma \ref{lem:mom:bnds:1}.

\begin{Lem}\label{l:emb}
For any $\gamma \leq \kappa^2 \nu/(4 C\|\sigma\|^2)$ and $K>0$,
\begin{align*}
\Prb \left( \sup_{t \geq 0} \left(\frac{\varepsilon}{2}  \|\U(t)\|^2 + \frac{\delta}{2}  \|\B(t)\|^2
+  \frac{C}{2\kappa \nu}  \| \Th(t)\|^2 +
\int_0^t \frac{\nu}{2} \|\nabla \U\|^2 + \|\nabla \B\|^2 +
\frac{C}{2\nu} \|\nabla \Th\|^2 \, ds \right.\right.  \\
 \left. \left.\vphantom{\int_0^t}
-
\frac{\varepsilon}{2}  \|\U(0)\|^2 - \frac{\delta}{2}  \|\B(0)\|^2
-  \frac{C}{ \kappa \nu}  \| \Th(0)\|^2
 - \frac{C \|\sigma\|^2}{\kappa \nu} t \geq K \right| \mathcal{F}_0
 \right) \leq e^{- \gamma K} \,,
\end{align*}
where $C = C(\TT^3)$.
\end{Lem}

We have the following corollary as an immediate consequence:

\begin{Cor}
\label{c:emb}
For any $T \geq 0$, $\delta, \varepsilon \leq 1$ and $\gamma$, $C$ as in
Lemma \ref{l:emb}, there exists
$\eta_0=\eta_0(\kappa,\nu, \|\sigma\|)$ such that for any $0<\eta<\eta_0$ and any $\mathcal{F}_0$ measurable $\U(0), \B(0), \Th(0)$
one has
\begin{align*}
\E \exp \left( \eta \sup_{t \in [0, T]} \left(\frac{\varepsilon}{2}  \|\U(t)\|^2 + \frac{\delta}{2}  \|\B(t)\|^2
+  \frac{C}{ 2\kappa \nu}  \|\Th(t)\|^2\right) +
\eta
\int_0^T \frac{\nu}{2} \|\nabla \U\|^2 + \|\nabla \B\|^2 +
\frac{C}{2\nu} \|\nabla \Th\|^2 \, ds \right) \\
\leq C \E \exp \left(\eta\left(
\frac{\varepsilon}{2}  \|\U(0)\|^2 + \frac{\delta}{2}  \|\B(0)\|^2
+  \frac{C}{ \kappa \nu}  \| \Th(0)\|^2
 + \frac{C \|\sigma\|^2}{\kappa \nu} T \right) \right) \,.
\end{align*}
\end{Cor}

\begin{proof}[Proof of Lemma \ref{l:emb}]
Testing \eqref{eq:MHD:vel}, \eqref{eq:MHD:mag} by $\U$ and $\B$ respectively and adding them we obtain
\begin{equation}\label{aub}
\frac{d}{dt} \left( \frac{\varepsilon}{2}  \|\U\|^2 + \frac{\delta}{2}  \|\B\|^2 \right) + \nu \|\nabla \U\|^2 + \|\nabla \B\|^2
= \langle \Th, \U_3\rangle \leq \frac{\nu}{2} \|\nabla \U\|^2 +
\frac{C}{\nu} \|\nabla \Th\|^2 \,.
\end{equation}
It\= o formula applied to  \eqref{eq:MHD:temp} yields
\begin{equation}\label{ave}
\frac{1}{2} d \|\Th\|^2 + \kappa \|\nabla \Th\|^2 dt
= \frac{1}{2}\|\sigma\|^2 dt + \langle \sigma, \Th \rangle dW \,.
\end{equation}
By adding $2C/(\kappa \nu)$ multiple of \eqref{ave} to \eqref{aub}
we obtain
\begin{align}
d \left( \frac{\varepsilon}{2}  \|\U\|^2 + \frac{\delta}{2}  \|\B\|^2
+  \frac{C}{ \kappa \nu}  \|\Th\|^2 \right)
+ \frac{\nu}{2} \|\nabla \U\|^2 + \|\nabla \B\|^2 +
\frac{C}{\nu} \|\nabla \Th\|^2
\leq \frac{C}{\kappa \nu} \|\sigma\|^2 dt +
\frac{2C}{\kappa \nu} \langle \sigma, \Th \rangle dW \,.
   \label{avebuth}
\end{align}
Using that for every $\gamma \leq \kappa^2 \nu/(4C\|\sigma\|^2)$ one has
for the quadratic variation of the Martingale
\begin{align*}
\frac{\gamma}{2} \frac{4C^2}{\kappa^2 \nu^2} |\langle \sigma, \Th \rangle|^2 \leq \frac{\gamma}{2} \frac{4C^2}{\kappa^2 \nu^2}
\|\sigma\|^2 \|\Th \|^2 \leq \frac{C}{2\nu} \|\nabla \Th \|^2 \,,
\end{align*}
and consequently the following exponential Martingale inequality: given a continuous Martingale $N=N(t)$,
for any $K \geq 0$,
\begin{equation} \label{eq:exp:mar}
\Prb\Big(\sup_{t \geq 0} \left(N(t) - \frac{\gamma}{2}\langle N(t)\rangle \right)
\geq K \Big| \mathcal{F}_0\Big) \leq \exp( - \gamma K) \,,
\end{equation}
where $\langle N\rangle$ is the quadratic
variation of $N$ implies the desired result.
\end{proof}

\begin{Lem}\label{l:imb}
Let $\alpha := \frac{\nu}{2\varepsilon} \wedge \frac{1}{\delta}$
and $\beta := \kappa/C$ (where $C = C(\TT^3)$ will be determined below).
There exists $\eta_1= \eta_1(\kappa,\nu, \|\sigma\|)$ such that for any $0<\eta<\eta_1$, $T > 0$, and $0<\eps,\del \leq c_0(\kappa, \nu, \TT^3)$,
\begin{align*}
\E &\exp \Big( \eta  \Big(\frac{\varepsilon}{2}  \|\U(T)\|^2 + \frac{\delta}{2}  \|\B(T)\|^2
+  \frac{C}{ \kappa \nu}  \| \Th(T)\|^2
 \notag \\
&\quad \quad\quad+
\exp(-\alpha T)\int_0^T \big(\frac{\nu}{2} \|\nabla \U\|^2 + \|\nabla \B\|^2\big)\,ds +
\exp(-\beta T)\frac{C}{2\nu}\int_0^T \|\nabla \Th\|^2 \, ds
 \Big)\Big)\notag \\
&\leq C\E \exp \left(\eta \left(
\frac{\varepsilon}{2} \exp(-\alpha T) \|\U(0)\|^2 + \frac{\delta}{2} \exp(-\alpha T) \|\B(0)\|^2
+  \frac{C}{ \kappa \nu} \exp(-\beta T) \| \Th(0)\|^2
 + \frac{C \|\sigma\|^2}{\beta\kappa \nu}  \right) \right) \,.
\end{align*}
\end{Lem}

\begin{proof}
Fix $T \geq 0$ and
for any real $r$ denote $m_r(t) := \exp(r(t- T))$. Then
from \eqref{aub} we obtain
\begin{align}
\label{eq:exp:1}
\frac{d}{dt} &\left( \frac{\varepsilon}{2} m_\alpha \|\U\|^2 + \frac{\delta}{2} m_\alpha \|\B\|^2 \right) + \nu m_\alpha \|\nabla \U\|^2 + m_\alpha \|\nabla \B\|^2
= m_\alpha \langle \Th, \U_3\rangle +
\frac{\varepsilon \alpha}{2} m_\alpha \|\U\|^2 +
\frac{\delta \alpha}{2} m_\alpha \|\B\|^2
\notag \\
&\leq \frac{\nu}{2} m_\alpha \|\nabla \U\|^2 +
\frac{1}{2} m_\alpha \|\nabla \B\|^2 +
\frac{C}{\nu} m_\alpha \|\nabla \Th\|^2 \,.
\end{align}
Also, from \eqref{ave} with $\beta = \kappa/C$ with $C = C(\TT^3)$
 we obtain
\begin{align}
\label{eq:exp:2}
\frac{1}{2} d (m_\beta\|\Th\|^2) + \kappa m_\beta \|\nabla \Th\|^2 dt
&= \frac{1}{2} m_\beta \|\sigma\|^2 dt +
\frac{\beta}{2} m_\beta \|\Th\|^2 dt +  m_\beta\langle \sigma, \Th \rangle dW
\notag \\
&\leq \frac{1}{2} m_\beta \|\sigma\|^2 dt +
\frac{\kappa}{2} m_\beta \|\nabla \Th\|^2 dt +  m_\beta\langle \sigma, \Th \rangle dW
\,.
\end{align}
If $\eps,\del\ll 1$, then $\beta \leq \alpha$ and $m_\beta \geq m_\alpha$ on $(-\infty, T]$. Then adding $2C/(\kappa \nu)$ multiple of
\eqref{eq:exp:2} to
\eqref{eq:exp:1} we find for any
 $\gamma \leq \frac{C(\TT^3) \kappa^2 \nu}{2\|\sigma\|^2}$,
\begin{align}
\label{eq:exp:3}
d &\left(\frac{\varepsilon}{2} m_\alpha \|\U\|^2 + \frac{\delta}{2} m_\alpha \|\B\|^2 + \frac{C}{\kappa \nu}m_\beta\|\Th\|^2\right)
\notag \\
&
\quad \quad \quad \quad + \left(\frac{\nu}{2}m_{\alpha}\|\nabla \U\|^2 + \frac{1}{2}m_{\alpha}\|\nabla \B\|^2
+ \frac{C}{2\nu} m_\beta \|\nabla \Th\|^2 \right)dt
\notag \\
&\quad \leq
\frac{C}{\kappa \nu} m_\beta \|\sigma\|^2 dt +
\frac{2}{\nu\kappa}m_\beta\langle \sigma, \Th \rangle dW -
\frac{\gamma}{2}
\frac{4}{\nu^2\kappa^2}m_{\beta}^{2}|\langle \sigma, \Th \rangle|^2dt
\,.
\end{align}
By the exponential Martingale inequality \eqref{eq:exp:mar},
for any $K>0$,
\begin{align*}
\Prb &\Bigg(\frac{\varepsilon}{2}  \|\U(T)\|^2 + \frac{\delta}{2}  \|\B(T)\|^2
+  \frac{C}{ \kappa \nu}  \| \Th(T)\|^2 \\
&\quad\quad +
\exp(-\alpha T)\int_0^T \big(\frac{\nu}{2} \|\nabla \U\|^2 + \|\nabla \B\|^2\big)\,ds +
\exp(-\alpha T)\frac{C}{2\nu}\int_0^T\|\nabla \Th\|^2 \, ds
\\ &\quad\quad-
\vphantom{\int_0^T}\frac{\varepsilon}{2} \exp(-\alpha T) \|\U(0)\|^2 - \frac{\delta}{2} \exp(-\alpha T) \|\B(0)\|^2
- \frac{C}{ \kappa \nu} \exp(-\beta T) \|\Th(0)\|^2
 - \frac{C \|\sigma\|^2}{\beta\kappa \nu} \geq K
\Bigg|\mathcal{F}_0\Bigg) \\
&\quad \leq e^{-\gamma K}.
\end{align*}
The proof follows by combining this estimate with the layer cake formula
$\E (|X|) = \int_0^\infty \Prb(|X| \geq L) \, dL$ for any random variable $X$.
\end{proof}

\section*{Acknowledgments}
This work was partially supported by the National Science Foundation under
the grants NSF-DMS-1207780 (SF) and NSF-DMS-1313272 (NEGH). We express our
appreciation and acknowledge support from the Mathematical Sciences
Research Institute (MSRI), the Mathematisches Forschungsinstitut
Oberwolfach (MFO) and the WISE Program at USC.  We would also like to thank
Jared Whitehead for providing helpful feedback and references.

\begin{footnotesize}

\end{footnotesize}

\vspace{.3in}
\begin{multicols}{2}
\noindent
Juraj F\"oldes\\
{\footnotesize
Deparment of Mathematics\\
Universit\'e Libre de Bruxelles\\
Web: \url{http://homepages.ulb.ac.be/~jfoldes/}\\
 Email: \url{juraj.foldes@ulb.ac.be}} \\[.2cm]
Susan Friedlander\\ {\footnotesize
Department of Mathematics\\
University of Southern California\\
Web: \url{http://www-bcf.usc.edu/~susanfri/}\\
 Email: \url{susanfri@usc.edu}}

\columnbreak

\noindent Nathan Glatt-Holtz\\ {\footnotesize
Department of Mathematics\\
Virginia Polytechnic Institute and State University\\
Web: \url{www.math.vt.edu/people/negh/}\\
 Email: \url{negh@math.vt.edu}} \\[.2cm]
Geordie Richards\\
{\footnotesize
Department of Mathematics\\
University of Rochester\\
Web: \url{www.math.rochester.edu/grichar5/}\\
 Email: \url{g.richards@rochester.edu}}\\[.2cm]

\end{multicols}

\end{document}